\documentclass[10pt, a4paper]{article}
\usepackage[T1]{fontenc}
\usepackage{geometry}

\usepackage[german, english]{babel}
\usepackage[utf8]{inputenc}
\usepackage{amssymb}
\usepackage{amsmath}
\usepackage{amsthm}
\usepackage{yfonts}
\usepackage{dsfont}
\usepackage{faktor}
\usepackage[pdftex]{graphicx}
\usepackage[arrow, matrix, curve]{xy}
\usepackage{nicefrac}
\usepackage{stmaryrd}
\usepackage{hyperref}
\usepackage{tikz}
\usepackage{nth}
\usepackage{xcolor}

\usetikzlibrary{arrows}
\newcommand{\midarrow}{\tikz \draw[->, thick] (0,0) -- +(.1,0);}

\newtheorem{num}{Nummerierung}[subsection]
\newtheorem{defin}[num]{Definition}
\newtheorem{prop}[num]{Proposition}
\newtheorem{lem}[num]{Lemma}

\newtheorem{cor}[num]{Corollary}
\newtheorem{rem}[num]{Remark}
\newtheorem{ex}[num]{Example}
\newtheorem{thm}[num]{Theorem}
\newtheorem*{thm*}{Theorem}
\newtheorem*{defin*}{Definition}
\newtheorem*{prop*}{Proposition}
\newtheorem*{lem*}{Lemma}
\newtheorem*{cor*}{Corollary}
\newtheorem*{rem*}{Remark}
\newtheorem*{ex*}{Example}
\newcommand{\N}{\mathbb{N}}

\newcommand{\R}{\mathbb{R}}
\newcommand{\C}{\mathbb{C}}

\newcommand{\T}{\mathbb{T}}
\newcommand{\cA}{\mathcal{A}}
\newcommand{\cB}{\mathcal{B}}

\newcommand{\cH}{\mathcal{H}}

\newcommand{\cK}{\mathcal{K}}
\newcommand{\cM}{\mathcal{M}}

\newcommand{\cR}{\mathcal{R}}
\newcommand{\cO}{\mathcal{O}}
\newcommand{\cS}{\mathcal{S}}

\newcommand{\g}{\mathfrak{g}}
\newcommand{\h}{\mathfrak{h}}
\newcommand{\n}{\mathfrak{n}}
\renewcommand{\a}{\mathfrak{a}}
\renewcommand{\t}{\mathfrak{t}}
\renewcommand{\k}{\mathfrak{k}}

\newcommand{\s}{\mathfrak{s}}
\renewcommand{\b}{\mathfrak{b}}
\newcommand{\z}{\mathfrak{z}}
\renewcommand{\l}{\mathfrak{l}}

\renewcommand{\sp}{\mathfrak{sp}(V,\Omega)}

\newcommand{\hsp}{\mathfrak{hsp}(V,\Omega)}

\newcommand{\spann}{\text{span}}

\newcommand{\ostrip}{\overline{\cS_\beta}}

\newcommand{\difftev}{\left.\frac{d}{dt}\right\vert_{t=0}}

\newcommand{\del}{\partial}
\newcommand{\tr}{\text{tr}}

\newcommand{\ad}{\text{ad}}
\newcommand{\tri}{\trianglelefteq}

\newcommand{\Aut}{\text{Aut}}
\newcommand{\qand}{\quad\text{and}\quad}
\newcommand{\qfor}{\quad\text{for}\quad}

\newcommand{\Wa}{\text{W}(\text{E},\sigma)}
\newcommand{\Hg}{\text{Heis}(\text{E},\sigma)}

\title{Factorial type I KMS states of Lie groups}
\author{Tobias Simon}

\begin{document}\selectlanguage{english}
\maketitle
\pagenumbering{Alph}
\renewcommand{\thepage}{C-\Roman{page}}

\pagenumbering{arabic}
\begin{abstract}
Motivated by the study of KMS conditions for C*- or W*-dynamical systems defined by covariant unitary representations of topological groups, we consider Gibbs states of a finite-dimensional Lie group $G$ and prove that these are precisely the factorial type I KMS states. For an element $X\in \textbf{L}(G)$ and an irreducible unitary representation $\rho$ of $G$ satisfying $\tr(e^{i\del\rho(X)})=1$, the corresponding Gibbs state is defined as $\varphi(g)=\tr(\rho(g)e^{i\del\rho(X)})$. We prove that under the mild assumption that $\rho$ has discrete kernel, the condition $\tr(e^{i\del\rho(X)})<~\infty$ implies that the generator $X$ is an inner point of the set $\text{comp}(\g)$ of elliptic elements in $\g$. This allows us to obtain a complete characterization of Lie algebras $\g$, representations $\rho$ with discrete kernel and generators $X$ such that $\tr(e^{i\del\rho(X)})<\infty$. 

\end{abstract}
\tableofcontents
\pagenumbering{arabic}
\section*{Notation and definitions}
\begin{itemize}
\item The group dynamical systems $(G,\alpha,\R)$ considered in this paper consist of a topological group $G$ and a group homomorphism $\alpha:\R\to\Aut(G)$ with continuous orbit maps, i.e. $\alpha$ is strongly continuous. We call $\alpha$ a one-parameter group of automorphisms.
\item A \textit{covariant unitary representation} of a group dynamical system $(G,\alpha,\R)$ is a unitary representation of $G^\flat:=G\rtimes_\alpha\R$. These are precisely the unitary representations $(\pi,\cH)$ of $G$ for which there exists a selfadjoint operator $H$ on $\cH$ which satisfies $e^{itH}\pi(g)e^{-itH}=\pi(\alpha_t(g))$, for all $t\in \R,\;g\in G$. We call $H$ an \textit{implementing operator}.
\item A \textit{state} $\varphi:G\to \C$ of a topological group $G$ is a continuous function, for which the kernel $K_\varphi:G\times G\to \C$ defined as $K_\varphi(x,y):=\varphi(x^{-1}y)$ is positive definite and $\varphi(e)=1$. Its GNS-representation is denoted by $(\pi_\varphi,\cH_\varphi,\Omega_\varphi)$.
\item If $G$ is a Lie group, we denote with $\textbf{L}(G)$ or $\g$ its Lie algebra.
\item A von Neumann algebra $\cM$ is uniquely decomposable as a direct sum of a type I, a type $\text{II}_1$, a type $\text{II}_\infty$ and a type III von Neumann algebra (cf. \cite[Def. V.1.17]{Tak02}, \cite[Thm. V.1.19]{Tak02}). It is called \textit{semi-finite} if the type III summand is trivial. This is equivalent to the existence of a semi-finite faithful normal tracial weight $\tau$ on $\cM$ (cf. \cite[Def. VII.1.1]{Tak03}, \cite[Thm. V.2.15]{Tak02}).
\item We call a state of a topological group or C*-algebra \textit{semi-finite} resp. \textit{factorial} if the von Neumann algebra generated by its GNS-representation is semi-finite resp. is a factor.
\item An \textit{elliptic} linear map endomorphism of a finite-dimensional vector space is a semi-simple linear map with purely imaginary eigenvalues.
 
\end{itemize}
\section*{Introduction}
For C*- or W*-dynamical systems $(\cM,\tau,\R)$, KMS states have been studied extensively, for instance in \cite{BR96}. We are interested in KMS states of dynamical systems defined by group dynamical systems $(G,\alpha,\R)$ and $\alpha$-covariant representations $\pi$, i.e. we consider the von Neumann algebra $\cM=\pi(G)^{\prime\prime}$ and the one-parameter group of *-automorphisms uniquely determined by $\tau_t(\pi(g))=\pi(\alpha_t(g))$. An example of such a dynamical system is the Weyl algebra $\Wa$ which is generated by the image of the Schrödinger representation of the Heisenberg group $\Hg$, where $(E,\sigma)$ is a symplectic vector space (cf. \cite[Thm. 5.2.8]{BR96}). In this example, the dynamics $\tau$ is implemented by Bogoliubov transformations corresponding to a one-parameter group of symplectic maps which is an example of the general situation described above. For such a dynamical system the $(\tau,\beta)$-KMS condition of a state $\omega:\cM\to \C$ proves that the corresponding state $\varphi:G\to \C,\;g\mapsto \omega(\pi(g))$ is an $(\alpha,\beta)$-KMS in the sense of the following definition. 
\begin{defin*}
\rm{Let $(G,\alpha,\R)$ be a group dynamical system and $\varphi:G\to \C$ be a state. For $\beta>0$, we say that $\varphi$ satisfies the $(\alpha,\beta)$-KMS condition if the functions 
\begin{align*}
F_{x,y}:\R\to \C,\quad t\mapsto \varphi(x\alpha_t(y)),\qfor x,y\in G,
\end{align*} 
extend to continuous maps on the strip $\ostrip:=\{z\in \C\;|\;0\leq \text{Im}(z)\leq \beta\}$ that are holomorphic on the interior and satisfy the reflection property $
F_{x,y}(t+i\beta)=\varphi(\alpha_t(y)x)$, for $t\in \R$.
}
\end{defin*}
If one conversely starts with an $(\alpha,\beta)$-KMS state of a topological group and considers the corresponding W*-dynamical system $(W^\ast(G),\tau,\R)$, where $W^\ast(G)$ is the group W*-algebra of $G$, one easily obtains a correspondence between $(\alpha,\beta)$-KMS states of $G$ and $(\tau,\beta)$-KMS states of $W^\ast(G)$ (cf. Subsection \ref{subselemprop}). In the case, where the group $G$ is locally compact, one has access to a C*-dynamical system given by group C*-algebra $C^\ast(G)$ defined as the enveloping C*-algebra of $L^1(G)$ and a corresponding action. This makes KMS states of locally compact groups more accessible since there exist stronger decomposition results for C*-dynamical systems and their KMS states. In particular, for locally compact type I groups, where every factor state is a type I state, every $(\alpha,\beta)$-KMS state is an integral of extremal Gibbs states with respect to some Radon probability measure on the set of extremal $(\alpha,\beta)$-KMS states (cf. Subsection \ref{subgroupc}).\\

If an $(\alpha,\beta)$-KMS state $\varphi$ of a topological group $G$ is semi-finite in the sense that $\pi_\varphi(G)^{\prime\prime}$ is a semi-finite von Neumann algebra, there exists a semi-finite faithful normal trace $\tau$ on $\pi_\varphi(G)^{\prime\prime}$. One can use the Radon-Nikodym theorem to prove that 
\begin{equation}\label{eqGibbs}
\varphi(g)= \frac{\tau(\pi(g)e^{-\beta H})}{\tau(e^{-\beta H})},\qfor g\in G,
\end{equation}
for a self-adjoint operator $H$ affiliated to $\pi_\varphi(G)^{\prime\prime}$ which satisfies $e^{itH}\pi_\varphi(g)e^{-itH}=\pi_\varphi(\alpha_t(g))$ and $\tau(e^{-\beta H})<\infty$ (cf. Subsection \ref{subssemif}). This applies in particular to the case, where $\pi_\varphi(G)^{\prime\prime}$ is a type I factor, i.e. $\pi_\varphi(G)^{\prime\prime}\cong B(\cH)$, for some Hilbert space $\cH$. Thus one obtains in the manner of Equation (\ref{eqGibbs}) a one-to-one correspondence between factorial type I $(\alpha,\beta)$-KMS states and unitary equivalence classes of irreducible $\alpha$-covariant unitary representations $\pi$ such that $\tr(e^{-\beta H})<\infty$ for the implementing operator $H$ of $\alpha$ which satisfies $e^{itH}\pi(g)e^{-itH}=\pi(\alpha_t(g))$.\\

Suppose that $G$ is a finite-dimensional Lie group and $\alpha$ is inner in the sense that there exists $X\in \textbf{L}(G)$ such that $\alpha_t$ is the conjugation by $\exp_G(tX)$. Then any unitary representation $\pi$ is $\alpha$-covariant since $e^{it\del\pi(X)}\pi(g)e^{-it\del\pi(X)}=\pi(\exp(tX)g\exp(-tX))$. This can always be achieved by considering, for not necessarily inner actions $\alpha$, the group $G^\flat:=G\rtimes_\alpha\R$, for which
\begin{align*}
(e,t)(g,s)(e,t)^{-1}=(\alpha_t(g),s),\qfor g\in G,\;s,t\in \R.
\end{align*} 
For an action by inner automorphisms, the representations relevant in the study of factorial type I KMS states are the following:
\begin{defin*}
\rm{Let $G$ be a finite dimensional Lie group and $X\in \textbf{L}(G)$. We call a unitary representation $(\pi,\cH)$ of $G$ of \textit{Gibbs type} with respect to $X$ if $e^{i\del\pi(X)}$ is of trace class and elements in $\textbf{L}(G)$, with respect to whom representations of Gibbs type exist, \textit{Gibbs elements}. We denote the set of Gibbs elements of $G$ with $\text{Gibbs}(G,\g)$  and for the universal covering group $\widetilde{G}$ of $G$ with $\text{Gibbs}(\g):=\text{Gibbs}(\widetilde{G},\g)$.}
\end{defin*}
Our main result is the following theorem, which turns out to be instrumental in characterizing representations of Gibbs type. 
\begin{thm*}
\it{Let $G$ be a finite dimensional Lie group and} \rm{$X\in \textbf{L}(G)$.} \it{If $(\pi,\cH)$ is a unitary representation of $G$ with discrete kernel such that $e^{i\del\pi(X)}$ is of trace class, then $X$ is contained in the interior of the set {\rm~$\text{comp}(\g)$} of elliptic elements in $\g$} \rm{(cf. Definition \ref{compemb})}. 
\end{thm*}
Irreducible unitary representations $\pi$ of finite-dimensional Lie groups, for which there exists  $X\in \text{comp}(\g)^\circ$ such that the spectrum of $i\del\pi(X)$ is bounded from above, where extensively studied in \cite{NE00}. Using these results, we obtain the following for irreducible representations of Gibbs type $\pi$ with discrete kernel, the Lie algebras $\g$, and Gibbs elements $X$:
\begin{itemize}
\item[(i)] $\g:=\textbf{L}(G)$ is an admissible Lie algebra (cf. Definition \ref{defadm}),
\item[(ii)] there exists a compactly embedded Cartan subalgebra $\t\subseteq \g$ containing $X$ (cf. Definition \ref{compemb}),
\item[(iii)] there exists an admissible positive system of roots $\Delta^+\subseteq \Delta(\g_\C,\t_\C)$ (cf. Definition \ref{rootadm}) such that $X$ is contained in the interior of the cone 
\begin{align*}
C_\text{max}(\Delta_p^+):=\{Y\in \t\;|\;i\alpha(X_0)\geq 0\;\text{for all}\;\alpha\in \Delta_p^+\},
\end{align*}
where $\Delta_p^+$ is the set of non-compact positive roots (cf. Definition \ref{defroots}),
\item[(iv)] $\pi$ is a unitary highest weight representation with respect to the root system $\Delta^+$ (cf. Definition \ref{unithigh}).
\end{itemize}
Conversely, if one starts with an admissible Lie group $G$ and a unitary highest weight representation $\pi$ of $G$ with respect to an admissible positive system $\Delta^+$, then \cite[Thm. IX.4.6]{NE00} yields that $e^{i\del\pi(X)}$ is of trace class for all $X\in C_\text{max}(\Delta_p^+)^\circ$. Hence we have characterized the Lie algebras, generators, and representations in question by geometric data. \cite[Thm. A.V.1]{NE00} gives a complete classification of simple admissible Lie algebras, which are also called hermitian. For these, the admissible root systems are known and the cones $C_\text{max}(\Delta_p^+)$ can be explicitly computed.\\

The condition of semi-boundedness of the representations relevant to this paper is generalized in \cite{Ni22} and shown to be a property of every KMS state of Lie groups. We also mention \cite{LT18} as an example of trace class conditions arising in conformal field theory.\\

\textbf{Structure of this paper}:\\
This paper is divided into two sections. What motivated this paper is mostly contained in Section~\ref{section1}. In this section, most of what we do is take results already known for KMS states of C*- or W*-dynamical systems and connect these with the group dynamical picture. This leads us to unitary representations of Gibbs type which we characterize in Section \ref{secGibbs}. This characterization can be seen as a standalone result which is why we include Section \ref{section1} afterward as a motivation.\\

\textbf{Acknowledgements}\\
This  research was supported by DFG-grant NE 413/10-1. I thank my supervisor Karl-Hermann Neeb for the many helpful corrections and Ricardo Correa da Silva for his advice on the theory of non-commutative integration.

\section{Geometric and representation theoretic aspects of Gibbs states}\label{secGibbs}
In this section, we study Gibbs elements and their corresponding unitary representations of Gibbs type. To do this, we make use of concepts introduced in Subsection \ref{subsadm} and Subsection \ref{subsrepr} in the appendix. These subsections are meant as a quick introduction to concepts discussed in depth in the monograph \cite{NE00}. 
\begin{defin*}
\rm{We call a subalgebra $\a\subseteq\g$ \it{compactly embedded}} \rm{if $\overline{\langle e^{\ad_\g(\a)}\rangle}\subseteq \text{GL}(\g)$ is compact. We call an element $x\in \g$ \textit{elliptic} if the subalgebra $\R\cdot x\subseteq \g$ is compactly embedded and denote with $\text{comp}(\g)$ the subset of elliptic elements in $\g$.}
\end{defin*}
Under the assumption that $\pi$ is an irreducible unitary representation of Gibbs type, we prove that the generator $X$ is contained in the interior of the set of elliptic elements $\text{comp}(\g)$ in $\g$. The condition $\text{comp}(\g)^\circ \neq \emptyset$ is equivalent to the existence of a compactly embedded Cartan subalgebra $\t\subseteq \g$, and in this case $\text{comp}(\g)=\text{Inn}(\g).\t$ (cf. \cite[Thm. VII.1.8]{NE00}). In Remark \ref{type}, we discuss how, for a compactly embedded Cartan subalgebra $\t\subseteq \g$, one obtains a root system 
\begin{align*}
\Delta:=\Delta(\g_\C,\t_\C)\subseteq i\t^\ast\quad\text{such that}\quad \g_\C=\t_\C\oplus \bigoplus_{\alpha\in \Delta}\g_\C^\alpha
\end{align*}
is the simultaneous eigenspace decomposition with respect to $\t$. In view of $\text{comp}(\g)=\text{Inn}(\g).\t$ and the $\text{Inn}(\g)$-invariance of the set of Gibbs elements, it suffices to characterize the elements of $\text{Gibbs}(\g)\cap\t$. We show that this characterization is possible by considering data associated to the root system $\Delta$, namely we prove that for a unitary representation $\pi$ of $G$ with discrete kernel and an element $X\in \g$ satisfying $e^{i\del\pi(X)}$, one has the following:
\begin{itemize}
\item[(i)] $\g$ is an admissible Lie algebra, i.e. it contains an invariant generating closed convex subset containing no affine line.
\item[(ii)] There exists a compactly embedded Cartan subalgebra $\t\subseteq \g$ containing $X$.
\item[(iii)] There exists an admissible positive system of roots $\Delta^+\subseteq \Delta(\g_\C,\t_\C)$ (cf. Definition \ref{rootadm}) such that $X$ is contained in the interior of the cone $C_\text{max}(\Delta_p^+)$ (cf. Definition \ref{cone}).
\item[(iv)] $\pi$ is a unitary highest weight representation with respect to the root system $\Delta^+$ (cf. Definition \ref{unithigh}).
\end{itemize}
In \cite[Thm. IX.4.6]{NE00} it is shown that if $\pi$ is a unitary highest weight representation with respect to an admissible positive system $\Delta^+$ and $X\in C_\text{max}(\Delta_p^+)^\circ$, then $e^{i\del\pi(X)}$ is of trace class. If one now drops the assumption of the discrete kernel, one obtains a complete characterization of Gibbs representations of simply connected Lie groups and Gibbs elements by considering faithful unitary highest weight representations with respect to admissible positive systems of admissible quotients.
\subsection{Ellipticity of Gibbs elements}
The goal of this subsection is to prove that, if $\pi$ is a unitary representation of a finite-dimensional Lie group $G$ with discrete kernel and $X\in \g:=\textbf{L}(G)$ such that $e^{i\del\pi(X)}$ is of trace class, then $X\in\text{comp}(\g)^\circ$. This is equivalent to $\g^X:=\text{ker}(\ad(X))\subseteq \g$ being compactly embedded by \cite[Lemma VII.1.7(c)]{NE00}.\\

The first crucial observation is that in this case, the eigenspaces of $e^{i\del\pi(X)}$ are all finite-dimensional and invariant under $G^X:=\langle \exp(\g^X)\rangle$. Since $\pi$ has discrete kernel, this implies that $\g^X$ is a compact Lie algebra. In order to prove that it actually is compactly embedded, we consider the smallest ideal $\n_Y\tri\g$, for a fixed $Y\in \g$, such that 
\begin{align*}
\ad_{\g/\n_Y}:\g/\n_Y\to \g/\n_Y,\; x+\n_Y\mapsto[Y,x]+\n_Y
\end{align*}
is an elliptic endomorphism. For $Y\in \g^X$, we prove that $\n_Y=\{0\}$ holds, which is equivalent to $Y$ being elliptic. To do this, we proceed by first showing that $\n_Y$ is abelian and then use this to prove that $\n_Y$ is an at most one-dimensional central ideal. 
\begin{lem}\label{nyab}
Let $(\pi,\cH)$ be a unitary representation of $G$ with discrete kernel such that $e^{i\del\pi(X)}$ is of trace class. The ellipticity ideal $\n_Y$ is abelian, for every $Y\in \g^X$.
\end{lem}
\begin{proof}
Consider the unitary representation of $G$ on the Hilbert-Schmidt operators $B_2(\cH)$ given by
\begin{align*}
\rho:G\to U(B_2(\cH)),\quad \rho(g)(A)=\pi(g)A\pi(g)^\ast.
\end{align*}
The eigenspace projections of $e^{i\del\pi(X)}$ have finite rank and thus are Hilbert-Schmidt operators. Since $\pi(\exp(\R Y))$ leaves the eigenspaces of $e^{i\del\pi(X)}$ invariant, the one-parameter groups $\rho(\exp(\R Y))$ fix the corresponding projections. Now Moore's Theorem \cite[Thm. 1.1]{Mo80} implies that the normal subgroup $N_Y:=\langle\exp(\n_Y)\rangle\tri G$ also fixes the projections onto the eigenspaces, which is equivalent to $\pi(N_Y)$ leaving the eigenspaces invariant. Hence the representation $\pi\vert_{N_Y}$ decomposes into finite-dimensional subrepresentations and has discrete kernel. Therefore the finite-dimensional unitary representations of $\n_Y$ separate the points which implies that $\n_Y$ is a compact Lie algebra. As a compact Lie algebra, $\n_Y$ is in particular reductive and decomposes as $\n_Y=\mathfrak{z}_{\n_Y}\oplus[\n_Y,\n_Y]$. We define $\a:=[\n_Y,\n_Y]\tri\g$  and note that $\a$ is a semi-simple compact ideal. From \cite[Cor. 5.6.14]{HN12}, one obtains that $\a$ is contained in a Levi complement $\s$ and thus for the ideal $\a\tri\s$ there exists a complementary ideal $\a^\bot\tri \s$ such that $\s=\a\oplus\a^\bot$ since $\s$ is semi-simple. It is now easy to see that $\b:=\s^\bot\oplus\mathfrak{rad}(\g)$ is an ideal in $\g$ and satisfies $\g=\a\oplus\b$. This shows that $\g/\b\cong \a$ is compact and thus $\ad_{\g/\b}(Y)$ is elliptic which implies $\n_Y\subseteq \b$ as $\n_Y$ is the smallest ideal with this property. Then $\a\subseteq\n_Y$ and $\a\cap\b=\{0\}$ imply $[\n_Y,\n_Y]=\a=\{0\}$ and thus the assertion follows. 
\end{proof}
The group $K:=\overline{\pi(N_Y)}\subseteq U(\cH)$ on $\cH$ leaves the finite dimensional eigenspaces of $e^{i\del\pi(X)}$ invariant, which implies that that $K$ is compact and Lemma \ref{nyab} shows that $K$ is abelian. Since $N_Y\tri G$ is a normal subgroup and $\pi(N_Y)\subset K$ is dense, it follows that $\pi(g)K\pi(g)^{-1}\subseteq K$ and thus
\begin{align*}
\Phi:G\to \text{Aut}(K),\quad \Phi(g)(k)= \pi(g)k\pi(g)^{-1}
\end{align*} 
is a group homomorphism. Next we consider on $\text{Aut}(K)$ a suitable topology\footnote{Consider the subspace topology $\mathcal{O}_1$ with respect to the topology of uniform convergence on $C(K,K)$, which is the coarsest topology containing the sets $
U(V_0,f_0):=\{f\in C(K,K)\;|\;f(k)\in V_0f_0(k)\;\text{for all}\; k\in K\},$ for $f_0\in C(K,K)$ and identity neighbourhoods $V_0$ in $K$. Let $\mathcal{O}_2$ be the unique topology on $\Aut(K)$ such that $\iota:\text{Aut}(K)\to (\text{Aut}(K), \mathcal{O}_1),\; f\mapsto f^{-1}$ becomes a homeomorphism. If we now define $\mathcal{O}:=\mathcal{O}_1\vee \mathcal{O}_2$, then $(\text{Aut}(K),\mathcal{O})$ defines a topological group (cf. \cite[p. 257]{HM06}).} for $\Phi$ to be continuous. We will use this topology to cite results that imply that a continuous action of a connected group acting on a compact abelian group by group homomorphisms is trivial.
\begin{lem}\label{Kcomm}
Let $(\pi,\cH)$ be a unitary representation of $G$ with discrete kernel such that $e^{i\del\pi(X)}$ is of trace class. For $Y\in \g^X$ and $N_Y:=\langle \exp(\n_Y)\rangle$, the group $K:=\overline{\pi(N_Y)}$ is compact, abelian, and commutes with the identity component of $G$. In particular, $\n_Y$ is central in $\g$.
\end{lem}
\begin{proof}
\cite[Thm. 9.82]{HM06} implies that the identity component of $\text{Aut}(K)$ with respect to the topology introduced in the footnote above coincides with the inner automorphisms of $K$. Since $K$ is abelian, it is trivial. As $\Phi$ is a group homomorphism, we only need to show that it is continuous with respect to $\cO_1$ to prove its continuity with respect to $\cO$. By the exponential law for spaces of continuous functions (cf.  \cite[Thm. 1.7.29]{GN}), this follows from the continuity of
\begin{align*}
\Phi^\wedge:G\times K\to K,\quad (g,k)\mapsto \Phi(g,k)=\pi(g)k\pi(g)^{-1}.
\end{align*}
This implies that the identity component of $G$ acts trivially on $K$ and therefore $\n_Y$ is central. The representation $\pi$ was assumed to be irreducible, so that $\pi(N_Y)\subseteq \pi(G)^\prime =\C\mathds{1}$ and the discreteness of the kernel of $\pi$ thus implies that $\n_Y\tri\g$ is a central ideal. With the same argument, one obtains that $\z(\g)$ is at most one-dimensional.
\end{proof}
It remains to prove that $\n_Y$ cannot be one-dimensional. In order to prove this, we will assume the opposite $\{0\}\neq \n_Y\subseteq \z:=\z(\g)$. Then we are in the context of the following lemma which will allow us to use some results of the representation theory of the three-dimensional Heisenberg group to obtain a contradiction.
\begin{lem}\label{heis}
Let $\g$ be a real finite-dimensional Lie algebra with center $\z$ and suppose that $y+\z\in \g/\z$ is elliptic but $y\in \g$ is not. Then there exists $x\in \g$ such that $[x,y]=z\in \z\setminus \{0\}$. In particular \rm{$\h:=\text{span}_\R\{x,y,z\}$} \it{is isomorphic to the three dimensional Heisenberg algebra $\h_3(\R)$.}
\end{lem}
\begin{proof}
We consider $\ad(y)$ as an endomorphism of $\g_\C$ with Jordan decomposition given by $\ad(y)=\ad(y)_s+\ad(y)_n.$ Since $\g_\C$ decomposes into the $\ad(y)_n$-invariant eigenspaces $\g_\C^\lambda$ of $\ad(y)_s$, for $\lambda\in \C$, and $\z_\C\subseteq \g_\C^0$, it follows from the fact that $y+\z$ is elliptic that 
\begin{align*}
\ad(y)_n\vert_{\g_\C^\lambda}\equiv 0,\quad\text{for }\quad\lambda\neq 0,\quad\text{and}\quad\ad(y)_n(\g_\C^0)\subseteq \z_\C.
\end{align*}
If $\ad(y)_n$ were zero, then $\ad(y)=\ad(y)_s$ would be elliptic since $\z_\C\subseteq \g_\C^0$. This proves that there exists $x\in \g_\C^0$ such that $\ad(y)(x)=\ad(y)_n(x)\in \z_\C\setminus\{0\}$. If we write $x=x_0+ix_1\in \g_\C$, then $[x,y]\in \z_\C\setminus \{0\}$ implies that either $[x_0,y]\in \z\setminus\{0\}$ or $[x_1,y]\in\z\setminus\{0\}$ hold. This proves the assertion.
\end{proof}

\begin{rem}\label{heisrep}
\rm{Consider the three-dimensional Heisenberg group defined as follows:
\begin{align*}
\text{Heis}(\R^2):=\R^2\times \R\quad\text{with multiplication}\quad (a,b,c)(x,y,z):=(a+x,b+y,c+z+ay),
\end{align*}
for $a,b,c,x,y,z\in \R$. For irreducible unitary representations of $\text{Heis}(\R^2)$ the center $Z(\text{Heis}(\R^2))\cong \R$ acts by scalars and thus, if the center acts non-trivially, defines a central character. The Stone-von-Neumann Theorem (cf. \cite[Thm. 10.2.1]{DE09}) implies that this central character $\lambda:\R\to \T,\;x\mapsto e^{i\lambda x}$ uniquely determines the unitary representation and that the unitary representation is unitarily equivalent to the representation on $L^2(\R)$ defined by
\begin{align*}
\pi_\lambda(a,b,c)f(x):=e^{i\lambda(bx+c)}f(x+a)\quad\text{for}\quad a,b,c,x\in \R,\;f\in L^2(\R).
\end{align*}
Consider the automorphism $\Phi:\text{Heis}(\R^2)\to \text{Heis}(\R^2),\;(a,b,c)\mapsto(b,-a,c-ab)$, which fixes the center pointwise. We note that by the Stone-von-Neumann theorem $\pi_\lambda\cong \pi_\lambda\circ \Phi$ and $\Phi(0,t,0)=(t,0,0)$ holds, for $t\in \R$. This proves that the unitary one-parameter groups defined by 
\begin{align*}
\pi_\lambda(t,0,0)f(x)=f(x+t)\quad\text{and}\quad \pi_\lambda(0,t,0)f(x)=e^{it\lambda x}f(x),\quad\text{for}\quad x,t\in \R,\;f\in L^2(\R)
\end{align*}
are conjugate via $U(\cH)$. In particular,$\pi_\lambda(\R,0,0)$ and $\pi_\lambda(0,\R,0)$ do not have compact closure in $U(\cH)$ since the identical representation of both groups on $L^2(\R)$ clearly does not decompose into irreducible subrepresentations and the groups are conjugate.
}
\end{rem}
We now come to our first main result:
\begin{thm}\label{semisimple}
Let $(\pi,\cH)$ be a unitary representation of $G$ with discrete kernel such that $e^{i\del\pi(X)}$ is of trace class. Then $\pi$ is a direct sum of irreducible representations and $X$ is contained in \rm{$\text{comp}(\g)^\circ$}.
\end{thm}
\begin{proof}
Since $e^{i\del\pi(X)}\in \pi(G)^{\prime\prime}$ is a positive operator of trace class, it follows that $\pi(G)^{\prime}$ leaves the finite-dimensional eigenspaces invariant. As finite dimensional C*-algebras are isomorphic to a direct sum of matrix algebras, this implies that there exists a maximal family of pairwise disjoint minimal projections in $\pi(G)^\prime$ summing to the idenitfy in the strong operator topology. Therefore $\pi$ is a direct sum of irreducible representations $\pi_j$ with discrete kernel satisfying $\tr(e^{i\del\pi_j(X)})<\infty$. Suppose now that $\pi$ is irreducible and let $Y\in \g^X$. It follows from Lemma \ref{Kcomm} that $\n_Y\subseteq \z(\g)$ and thus $Y$ is either elliptic or $0\neq Y+\z(\g)$ is elliptic in $\g/\z(\g)$. In the latter case, Lemma \ref{heis} implies that there exists $x_0\in \g$ such that $[Y,x_0]=z\in \z(\g)\setminus\{0\}$ and thus the subalgebra $\h:=\text{span}_\R\{Y,x_0,z\}$ is isomorphic to the three-dimensional Heisenberg algebra $\textbf{L}(\text{Heis}(\R^2))$ via the isomorphism defined by
\begin{align*}
Y\mapsto (1,0,0),\;x_0\mapsto (0,1,0), z\mapsto (0,0,1).
\end{align*} 
The simply connected Heisenberg group $\text{Heis}(\R^2)$ is the universal covering group of $\langle\exp(\h)\rangle$ and thus the representation $\pi\vert_{\langle \exp(\h)\rangle}$ lifts to a representation $\rho$ of $\text{Heis}(\R^2)$. As $\pi$ is irreducible, $\z(\g)$ acts by scalars, i.e. there exists $\lambda \in \R $ such that 
\begin{align*}
\del\rho(0,0,1)=\del\pi(z)=i\lambda \mathds{1}.
\end{align*} 
One has $\lambda\neq 0$, because $\del\pi(z)=0$ would imply $
\n_Y\subseteq \z(\g)\subseteq \text{ker}(\del\pi)=\{0\}$. The Stone-von-Neumann Theorem (cf. \cite[Thm. X.3.1]{NE00}) implies that $\rho$ is a multiple of $\pi_\lambda$, where $\pi_\lambda$ is the unique irreducible unitary representation of the Heisenberg group with central character $\lambda:\R\to \T,\;x\mapsto e^{i\lambda x}$. It follows from Remark \ref{heisrep}, that the closure of 
\begin{align*}
\rho(\exp(\R(1,0,0)))=\pi(\exp(\R Y))
\end{align*} 
is not compact. This is a contradiction since a closed subgroup of $U(\cH)$ leaving the finite-dimensional eigenspaces of $e^{i\del\pi(X)}$ invariant is compact. Hence $\n_Y=\{0\}$ follows, which proves that $Y$ is elliptic. As $Y\in\g^X$ was arbitrary, we have shown that  $\g^X$ is contained in  $\text{comp}(\g)$, which together with \cite[Lemma VII.1.5(b)]{NE00} yields that $\g^X$ is compactly embedded. From \cite[Lemma VII.1.7(c)]{NE00}, one obtains that $X\in\text{comp}(\g)^\circ$.
\end{proof}
\subsection{Characterization of representations of Gibbs type}\label{subscharGibbs}
The goal of this subsection is to summarize results from \cite{NE00} to prove that the irreducible unitary representation $\rho$ of Gibbs type with discrete kernel coincide with faithful unitary highest weight representations of a finite-dimensional connected Lie group $G$. The Lie algebras, for which such representations exist, are precisely the admissible Lie algebras, i.e. those containing a generating invariant closed convex subset containing no affine line. The Lie algebraic and representation theoretic concepts used in this subsection are introduced in Subsection \ref{subsadm} and Subsection \ref{subsrepr}. \\

The convex momentum set $I_\pi\subseteq \g^\ast$ of a unitary representation $(\pi,\cH)$ of a finite-dimensional Lie group $G$ (cf. Definition \ref{convmom}) connects convex geometry and representation theory. Its relevance for our study of representations of Gibbs type can be seen from the following result (cf. \cite[Lemma X.I.6]{NE00}):
\begin{align*}
B(I_\pi):=\{Y\in \g\;|\; \langle I_\pi,Y\rangle >-\infty\}=\{Y\in \g\;|\;\sup \text{spec}(i\del\pi(Y))<\infty\}.
\end{align*}
If $(\pi,\cH)$ is an irreducible unitary representation of $G$ with discrete kernel for which there exists $X\in \g:=\textbf{L}(G)$ such that $e^{i\del\pi(X)}$ is of trace class, then the spectrum of $i\del\pi(X)$ is bounded from above and Theorem \ref{semisimple} implies $X\in B(I_\pi)\cap \text{comp}(\g)^\circ$. Therefore \cite[Thm. X.3.8]{NE00} is applicable and we obtain the following corollary:
\begin{cor}\label{corad}
Let $G$ be a finite-dimensional Lie group and $(\pi,\cH)$ be an irreducible unitary representation of $G$ with discrete kernel for which there exists $X\in \g:=\textbf{L}(G)$ such that $e^{i\del\pi(X)}$ is of trace class. Then there exist a compactly embedded Cartan subalgebra $\t\subseteq\g$ containing $X$ and a positive system of roots $\Delta^+\subseteq \Delta(\g_\C,\t_\C)$, which is adapted \rm{(cf. Definition \ref{defroots})}, \it{such that $\pi$ is a unitary highest weight representation} \rm{(cf. Definition \ref{unithigh})} \it{with respect to $\Delta^+$ and $X$ is contained in the interior of the cone}\rm{
\begin{align*}
C_\text{max}:=C_\text{max}(\Delta_p^+)=\{X\in \t\;|\; i\alpha(X)\geq 0\;\text{for all}\;\alpha\in \Delta_p^+\}.
\end{align*}}
\end{cor}
For an element $X\in \g$, the condition that there exists an  adapted positive system $\Delta^+$ such that $X\in C_\text{max}(\Delta_p^+)^\circ$ is not sufficient for $X$ to be a Gibbs element. To ensure this, we need further restrictions on $\Delta^+$. The Gelfand-Rai'kov Theorem for Unitary Highest Weight Representations \cite[Thm. X.4.8]{NE00} states sufficient conditions for a positive root system $\Delta^+$ of a simply connected admissible Lie group such that unitary highest weight representations exist. In this theorem, one considers the convex cone
\begin{align*}
C_\text{min}(\Delta_p^+):=\text{cone}(\{i[Z_\alpha,Z_\alpha^\ast]\;|\;Z_\alpha\in\g_\C^\alpha,\;\alpha\in\Delta_p^+\}\subseteq \t,
\end{align*}
where $\ast:\g_\C\to\g_\C,\;x+iy\mapsto -x+iy$ denotes the involution on $\g_\C$ discussed in Remark \ref{type}. The condition stated in the following definition then ensures that the unitary highest weight representations separate the points.
\begin{defin}\label{rootadm}
\rm{Let $\g$ be a Lie algebra and $\t\subseteq\g$ be a compactly embedded Cartan subalgebra. We call an adapted positive system of roots $\Delta^+\subseteq \Delta(\g_\C,\t_\C)$ such that $C_\text{min}(\Delta_p^+)$ is pointed and contained in $C_\text{max}(\Delta_p^+)$ \textit{admissible}.
}
\end{defin}
\begin{rem}
\rm{The term admissible for a positive system of roots is not used in \cite{NE00}. We introduce it to state the conditions for positive systems of roots arising in our context more concisely.
}
\end{rem}
The condition of admissibility of a positive system arises quite naturally in the theory of unitary highest weight representations. In the Gelfand-Raik'kov theorem, it is stated as a sufficient condition for the existence of unitary highest weight representation of admissible Lie algebras. The following theorem, which is stated implicitly in \cite{NE00}, shows that for faithful unitary highest weight representations the condition is also necessary:
\begin{thm}\label{admuhwstatement}\rm{(\cite[Thm. IX.2.17]{NE00})}\\
\it{Let $G$ be a finite dimensional Lie group with Lie algebra $\g$ and $(\pi,\cH)$ be an irreducible unitary highest representation of $G$ with respect to a positive system $\Delta^+$ and suppose $\pi$ has discrete kernel. Then $\g$ is admissible and $\Delta^+$ is an admissible positive system.}
\end{thm}
\begin{proof}
We consider the space of smooth vectors $\cH^\infty$ which is a faithful, unitary highest weight module of $\g_\C$. It follows from \cite[Prop. IX.1.14]{NE00}(iii) that, for the highest weight $\lambda\in i\t^\ast$ of $\pi$ the $\g_\C$-module $\cH^\infty$ is isomorphic to $L(\lambda,\Delta^+)$, which is introduced there. Using \cite[Thm. IX.2.17]{NE00}, which is applicable to $L(\lambda,\Delta^+)$, one obtains that $\g$ is admissible, $\Delta^+$ is adapted and $-i\lambda\in \t^\ast$ is contained in the interior of 
\begin{align*}
C_\text{min}(\Delta_p^+)^\star:=\{f\in \t^\ast \;|\; f(C_\text{min}(\Delta_p^+))\subseteq [0,\infty)\}\subseteq B(C_\text{min}(\Delta_p^+)).
\end{align*}
Then \cite[Prop. V.1.15]{NE00} implies that the convex cone $C_\text{min}(\Delta_p^+)$ is pointed since $B(C_\text{min}(\Delta_p^+))$ has inner points. To prove that $C_\text{min}(\Delta_p^+)\subseteq C_\text{max}(\Delta_p^+)$ holds, we apply \cite[Prop. VIII.1.11]{NE00} which is applicable if there exists $f\in C_\text{min}(\Delta_p^+)^\star$ such that its orbit $\mathcal{O}_{f}$ under the coadjoint action is closed, the convex hull of $\mathcal{O}_{f}$ contains no affine lines and $\mathcal{O}_f$ spans $\g^\ast$.\\

If one applies \cite[Thm. VIII.1.19]{NE00} to the highest weight $\lambda\in i C_\text{min}(\Delta_p^+)^\star \subseteq i\t^\ast$ of the unitary highest weight representation $\pi$ with discrete kernel, one obtains that $-i\lambda$ is admissible, i.e. its orbit $\mathcal{O}_{-i\lambda}$ under the coadjoint action is closed and its convex hull contains no affine lines. Using \cite[Thm. X.4.1]{NE00} which states that the momentum set of $\pi$ is given by $I_\pi=\text{conv}(\mathcal{O}_{-i\lambda})$ and $(I_\pi)^\bot=\text{ker}(d\pi)=\{0\}$ (cf. \cite[Lemma X.1.6]{NE00}), one obtains that $I_\pi\subseteq \g^\ast$ is generating and
\begin{align*}
-i\lambda \in C_\text{min}(\Delta_p^+)^\star\quad\text{satisfies}\quad\spann_\R\mathcal{O}_{-i\lambda}=\g^\ast.
\end{align*} 
Hence \cite[Prop. VIII.1.11]{NE00} is applicable which yields $C_\text{min}(\Delta_p^+)\subseteq C_\text{max}(\Delta_p^+)$ and thus the positive system $\Delta^+$ is admissible.
\end{proof}
\begin{rem}\label{remuniq}
\rm{Note that \cite[VIII.1.11]{NE00} implies that the positive system $\Delta^+$ in the theorem above is uniquely determined by the condition that $(\pi,\cH)$ is an irreducible unitary highest representation with respect to $\Delta^+$.}
\end{rem}
In \cite{NE00}, the highest weight modules $L(\lambda,\Delta^+)$ are studied, for involutive Lie algebras with root decomposition (cf. \cite[Subsection IX.2]{NE00}). For those modules, one has the following result:
\begin{thm}\label{IX.4.6}\rm{(\cite[Thm. IX.4.6]{NE00})}\\
\it{Let $\g$ be an admissible Lie algebra, $\Delta^+$ an adapted positive system and $L(\lambda,\Delta^+)$ a unitary highest weight module with respect to $\Delta^+$. If} \rm{$X\in C_{\text{max}}(\Delta_p^+)^\circ$} \it{and $\del\pi$ denotes the representation of $\g$ on $L(\lambda,\Delta^+)$, then} \rm{$\tr(e^{i\del\pi(X)})<\infty$.}

\end{thm}
\begin{rem}
\rm{Let $\g$ be an admissible Lie algebra. Then \cite[Thm. VII.3.10]{NE00} implies that there exists a compactly embedded Cartan subalgebra $\t$  in $\g$ and an admissible positive system of roots $\Delta^+$. Then \cite[Prop. VIII.3.7]{NE00} associates to each admissible positive system $\Delta^+$ a closed generating invariant cone $W_\text{max}(\Delta_p^+)\subseteq\g$ such that $ W_\text{max}\cap \t =C_\text{max}$. In the proof of \cite[Thm. VIII.3.10]{NE00} it is also shown that
\begin{align*}
W_\text{max}(\Delta_p^+)^\circ=\text{Inn}(\g).C_\text{max}(\Delta_p^+)^\circ\subseteq\text{comp}(\g)^\circ.
\end{align*}
If $\g$ is non-compact, the open invariant convex cone $W_\text{max}(\Delta_p^+)^\circ$ is proper since $\text{comp}(\g)^\circ \subsetneq \g$.
}
\end{rem}
We are now able to present our second main result:
\begin{thm}
Let $G$ be a finite-dimensional Lie group with Lie algebra $\g$ and $(\pi,\cH)$ be an irreducible unitary highest representation of $G$ with respect to a positive system $\Delta^+$. Then one has \rm{$\{X\in \g\;|\;\tr(e^{i\del\pi(X)})<\infty\}=W_\text{max}(\Delta_p^+)^\circ.$}
\end{thm}
\begin{proof}
This follows immediately from Corollary \ref{corad}, Theorem  \ref{admuhwstatement}, and Remark \ref{remuniq}.
\end{proof}
As the Gelfand-Rai'kov Theorem (cf. \cite[Thm. X.4.8]{NE00}) applies to simply connected admissible Lie groups, we may combine these results to obtain the following corollary:
\begin{cor}\label{class}
Let $G$ be a simply connected finite-dimensional Lie group with Lie algebra $\g$ and $X\in \g$. Then there is a one-to-one correspondence between 
\begin{itemize}
\item[\rm{(i)}] irreducible unitary representation $\pi$ of $G$ such that $e^{i\del\pi(X)}$ is of trace class,
\item[\rm{(ii)}] irreducible unitary highest weight representations of admissible quotients $q:\g \twoheadrightarrow\g_1$ with respect to an admissible positive system $\Delta_1^+$ of $\g_1$ such that \rm{$q(X)\in W_\text{max}(\Delta_{1,p}^+)^\circ$.}
\end{itemize}
\end{cor}
We will now prove that an element $X$ is contained in $\text{Gibbs}(\g)$ if and only if it is contained in an open proper invariant convex subset of $\g$. 
\begin{thm}\label{geom}
{\rm~\textbf{Geometric characterization of Gibbs elements}}\\
Let $\g$ be a finite-dimensional Lie algebra. Then an element $X$ is a Gibbs element if and only if it is contained in an open proper invariant convex subset of $\g$. 
\end{thm}
\begin{proof}
Suppose $X\in\text{Gibbs}(\g)$ and $G$ is a simply connected Lie group with Lie algebra $\g$. Then there exists a unitary representation $\pi$ of $G$ such that $e^{i\del\pi(X)}$ is of trace class. Then Corollary \ref{corad} and Theorem \ref{admuhwstatement} imply that $\g/\text{ker}(d\pi)$ is admissible and there exists an admissible positive system $\Delta^+$ such that 
\begin{align*}
q(X)\in C_\text{max}(\Delta_p^+)^\circ\subseteq W_\text{max}(\Delta_p^+)^\circ,
\end{align*}
where $q:\g\to \g/\text{ker}(d\pi)$ denotes the quotient map. If $\g/\text{ker}(d\pi)$ is non-compact, the subset 
\begin{align*}
q^{-1}(W_\text{max}(\Delta_p^+)^\circ)\subsetneq \g
\end{align*} 
is an open proper invariant convex subset containing $X$. If $\g^\prime:=\g/\text{ker}(d\pi)$ is compact, choose an $\text{Inn}(\g^\prime)$-invariant metric on $\g^\prime$ and an open ball $B\subsetneq \g^\prime$ with respect to this metric containing $q(X)$. Then $q^{-1}(B)\subsetneq \g$ is an open proper invariant convex subset containing $X$.\\

Suppose conversely that $X$ is contained in an open proper invariant convex subset $C\subsetneq \g$. We adopt elements of the proof of \cite[Prop. VII.3.4]{NE00}. We consider the set
\begin{align*}
H(C):=\{v\in \g\;|\;v+C=C\}.
\end{align*}
$H(C)$ is a convex additive subgroup of the vector space $\g$ which implies that $H(C)$ a linear subspace  and the invariance of $C$ yields that $H(C)$ is an ideal. Consider the quotient map $q:\g\to\g^\prime:= \g/H(C)$ and observe that $q(C)$ is an invariant open generating convex subset. It satisfies
\begin{align*}
H_{\g^\prime}(q(C))=\{w\in \g^\prime\;|\;w+q(C)=q(C)\}=\{q(v)\in \g^\prime\;|\;v\in \g,\;q(v+C)=q(C)\}=\{0\}
\end{align*}
since $q(v+C)=q(C)$ is equivalent to $(v+C)+H(C)=C+H(C)$, and then $C+H(C)=C$ implies $v\in H(C)$. Therefore the closure of $q(C)$ is an invariant generating closed convex subset containing no affine line and thus $\g^\prime$ is admissible. To apply Theorem \ref{IX.4.6} and Corollary \ref{class}, we show that there exists an admissible positive system $\Delta^+$ such that $q(X)\in C_\text{max}(\Delta_p^+)^\circ$.  \\

The set $q(C)$ is contained in the interior of its closure $\overline{q(C)}$ whose interior consists of elliptic elements (cf. \cite[Prop. VII.3.4(b)]{NE00}) and thus $q(X)$ is contained in the open invariant elliptic convex set $q(C)$.\\

For a compactly embedded Cartan subalgebra $\t\subseteq\g^\prime$, one has $\text{comp}(\g^\prime)=\text{Inn}(\g^\prime).\t$ (cf. \cite[Thm. VII.1.8]{NE00}), we may assume $q(X)\in \t$. Then the Sandwich Theorem (cf. Theorem \ref{thmsandwich}) implies that there exists an adapted positive system $\Delta^+$ with respect to $\t$ such that 
\begin{align*}
\overline{q(C)}\cap\t\subseteq C_\text{max}(\Delta_p^+)\quad\text{and}\quad C_\text{min}(\Delta_p^+)\subseteq \lim(\overline{q(C)}\cap \t):=\{v\in \t\;|\;v+\overline{q(C)}\cap\t\subseteq  \overline{q(C)}\cap\t\}.
\end{align*}
As $\overline{q(C)}$ contains no affine line, the cone $C_\text{min}(\Delta_p^+)$ is pointed and since $C_\text{max}$ is a cone and $\overline{q(C)}\cap\t\subseteq C_\text{max}(\Delta_p^+)$, it follows that 
\begin{align*}
C_\text{min}(\Delta_p^+)\subseteq\lim(\overline{q(C)}\cap \t)\subseteq C_\text{max}(\Delta_p^+).
\end{align*} 
This proves that $\Delta^+$ is admissible and that $q(X)\in C_\text{max}(\Delta_p^+)^\circ$. Hence the assertion follows.
\end{proof}

\begin{ex}
\rm{Let $(V,\Omega)$ be a symplectic vector space, $\h(V,\Omega)=\R\oplus V$ be the corresponding Heisenberg algebra and 
\begin{align*}
\mathfrak{sp}(V,\Omega):=\{A\in \mathfrak{gl}(V)\;|\;\Omega(v,Aw)+\Omega(Av,w)=0\;\text{for all}\; v,w\in V\}
\end{align*} 
be the corresponding symplectic Lie algebra. The prototypical example of an admissible Lie algebra which is not semi-simple, is given by the Jacobi algebra $\g:=\hsp:=\h(V,\Omega)\rtimes \sp$. The map
\begin{align*}
\Phi:\hsp\to C^\infty(V),\quad \Phi(s,v,A)(w):=\tfrac{1}{2}\Omega^{A}(w)+\Omega(v,w)+s,
\end{align*}
for $s\in \R,\;v,w\in V,\;A\in \sp$, defines an injective homomorphism of Lie algebras if one equips $C^\infty(V)$ with the Poisson-bracket corresponding to $\Omega$ (cf. \cite[Prop. A.IV.15]{NE00}). The image of $\Phi$ coincides with the polynomials of degree $\leq 2$ and it maps $\sp$ to polynomials of homogeneous degree $2$. Continuing onwards we identify $\hsp\cong \text{im}(\Phi)$. Consider a symplectic basis
\begin{align*}
\{p_j,q_j\;|\; 1\leq j\leq n:=\tfrac{1}{2}\text{dim}_\R V\}\quad\text{and define}\quad \t_\l:=\spann_\R\{p_i^2+q_i^2\;|\;1\leq i\leq n\}.
\end{align*}
Then $\t_\l$ is a compactly embedded Cartan subalgebra of $\l:=\sp$. If one computes the root system $\Delta:=\Delta(\l_\C,(\t_\l)_\C)$ as in \cite[Ex. VII.2.30]{NE00} and chooses the up to a sign unique positive admissible system  $\Delta^+$, then one has
\begin{align*}
C_\text{max}(\Delta_p^+)=\text{cone}(\{p_i^2+q_i^2\;|\;1\leq i\leq n\})=\{\Omega^{A}\in\t_\l\;|\; \Omega^{A}\geq 0\}.
\end{align*}
Since there is up to a sign only one admissible positive system of non-compact roots $\Delta_p^+$ in the simple hermitian Lie algebra $\sp$ (cf. \cite[Lemma VII.2.16(iv)]{NE00})  and $\Omega^{A}> 0$ if and only if $A$ is elliptic and all its eigenvalues are contained in $i\R^+$, this proves that the set of Gibbs elements consists of all elliptic symplectic endomorphisms for which the eigenvalues are contained in either $i\R^+$ or $i\R^-$. \\

Then \cite[Ex. VII.2.30]{NE00}) implies that $\t=\z(\g)\oplus \t_\l$ is a compactly embedded Cartan subalgebra of $\hsp$ for the corresponding admissible positive systems of roots $\Delta_0^\pm\subseteq\Delta(\g_\C,\t_\C)$ defined by extending the elements of $\Delta^\pm\subseteq \Delta(\l_\C,(\t_\l)_\C)$ extended to $\t=\z(\g)\oplus \t_\l$ by zero on $\z(\g)$. In particular, one has
\begin{align*}
W_\text{max}(\Delta_{0,p}^\pm)^\circ=\text{Inn}(\g).C_\text{max}(\Delta_{0,p}^\pm)^\circ=\text{Inn}(\g).(\z(\g)\oplus C_\text{max}(\Delta_{p}^\pm)^\circ)
\end{align*}
which implies $W_\text{max}(\Delta_{0,p}^\pm)^\circ=\{(s,v,A)\in \hsp\;|\;\Omega^{A}>0\}$ which gets mapped to the set of strictly convex polynomials of degree $\leq 2$ by $\Phi$. This is the interior of the $\text{Inn}(\g)$-invariant cone 
\begin{align*}
W:=\{f\in \text{im}(\Phi)\;|\;\inf f>\infty\}=\left\{\Phi(s,v,A)\;|\;\Omega^{A}\geq 0,\;v\in \text{rad}(\Omega^{A})^{\bot_\Omega}\right\}
\end{align*}
of polynomials of degree $\leq 2$ which are bounded from below. Note that this set is not closed as its boundary coincides with all polynomials of degree one.
}
\end{ex}
\section{KMS states of topological groups}\label{section1}
In this section, we discuss how results for KMS states of C*-algebras and W*-algebras can be transferred to topological groups. For general topological groups $G$, there exists a universal group W*-algebra $W^\ast(G)$. It has the property that, for a strongly continuous one-parameter group of automorphisms $\alpha:\R\to \Aut(G)$, one obtains a $\sigma$-weakly continuous one-parameter group of *-automorphisms $\tau:\R\to \Aut(W^\ast(G))$ such that a state $\varphi$ of $G$ is an $(\alpha,\beta)$-KMS state if and only if its corresponding state $\omega_\varphi$ of $W^\ast(G)$ is a $(\tau,\beta)$-KMS state (cf. Proposition \ref{kmsW*}). As is the case for W*-dynamical systems, one does not have good control over the set of KMS states in general. For locally compact groups this is not the case, as the availability of the group C*-algebra $C^\ast(G)$ has the consequence that the set $K_\beta$ of $(\alpha,\beta)$-KMS states of a locally compact topological group is a weak*-compact convex subset of $L^\infty(G)\cong L^1(G)^\prime$ (cf. Theorem \ref{convex}). Therefore the Krein-Millman Theorem implies that $K_\beta$ is the closed convex hull of its extremal points $\text{Ext}(K_\beta)$. As a KMS state is extremal if and only if it is a factor state, it suffices to study factorial KMS states. For general topological groups $G$, we characterize semi-finite KMS states as generalized Gibbs states. This allows us to apply the results from Section \ref{secGibbs} to factorial type I KMS states of finite-dimensional Lie groups $G$ for one-parameter groups of inner automorphisms generated by $X\in \textbf{L}(G)$, which correspond to irreducible unitary representations $\rho$ such that $\tr(e^{i\beta\del\rho(X)})<\infty$. 
\subsection{Elementary properties of KMS states of topological groups}\label{subselemprop}
The goal of this subsection is to relate the concepts $(\alpha,\beta)$-KMS states of a topological group $G$ and KMS states of W*-dynamical systems defined by $\alpha$-covariant representations of the group~$G$. In particular, we show that each $(\alpha,\beta)$-KMS state of a topological group is $\alpha$-invariant.
\begin{defin}\label{defkms}
\rm{(cf. \cite[Def. 5.3.1]{BR96}, \cite[Lemma 5.3.7]{BR96}):\\
Let $\omega$ be a state of a C*-algebra $\cA$ (resp. normal state of a von Neumann algebra) and let $\tau:\R\to\Aut(\cA)$ be a strongly (resp. $\sigma$-weakly) continuous one-parameter group of *-automorphisms. We call $\omega$ an $(\tau,\beta)$-KMS state if the functions 
\begin{align*}
F_{A,B}:\R\to \C,\quad t\mapsto \omega(A\tau_t(B)),\qfor A,B\in \cA,
\end{align*} 
extend to continuous maps on the strip $\ostrip:=\{z\in \C\;|\;0\leq \text{Im}(z)\leq \beta\}$ that satisfy the reflection relation $F_{A,B}(t+i\beta)=\omega(\tau_t(B)A)$, for $t\in \R$, and are holomorphic on the interior.}
\end{defin}
We mimic the equivalent definition of KMS states of C*-or W*-dynamical systems $(\cA,\tau,\R)$ to define KMS states of topological groups (cf. \cite[Lemma 5.3.7]{BR96}). 
\begin{defin}\label{defkmsgrp}
\rm{Let $(G,\alpha,\R)$ be a group dynamical system and $\varphi:G\to \C$ be a state. For $\beta>0$, we say that $\varphi$ satisfies the $(\alpha,\beta)$-KMS condition if the functions 
\begin{align*}
F_{x,y}:\R\to \C,\quad t\mapsto \varphi(x\alpha_t(y)),\qfor x,y\in G,
\end{align*} 
extend to continuous maps on the strip $\ostrip:=\{z\in \C\;|\;0\leq \text{Im}(z)\leq \beta\}$ that are holomorphic on the interior and satisfy the reflection relation $F_{x,y}(t+i\beta)=\varphi(\alpha_t(y)x)$, for $t\in \R$.}
\end{defin}
We now define a W*-dynamical system $(W^\ast(G),\tau,\R)$ corresponding to a group dynamical system $(G,\alpha,\R)$ and a state $\omega_\varphi$ of $W^\ast(G)$ corresponding to a state $\varphi$ such that $\varphi$ is an $(\alpha,\beta)$-KMS state of $G$ if and only if $\omega_\varphi$ is a $(\tau,\beta)$-KMS state of $W^\ast(G)$.
\begin{defin}
\rm{For a topological group $G$, we define the group W*-algebra as a pair of a von Neumann algebra $W^\ast(G)$ and a $\sigma$-weakly continuous group homomorphism $\pi:G\to U(W^\ast(G))$ that has the universal property that for every strongly continuous unitary representation $(\rho,\cH)$ of $G$ there exists a unique *-homomorphism 
\begin{align*}
\Phi:W^\ast(G)\to \rho(G)^{\prime\prime}\quad\text{such that}\quad \Phi(\pi(g))=\rho(g),\quad\text{for all}\quad g\in G.
\end{align*}
For every topological group $G$, there exists such an algebra by \cite[Prop. 4.2]{GN00}.
}
\end{defin}
\begin{rem}
\rm{Let $G$ be a topological group.
\begin{itemize}
\item[(i)] By the universal property, for every $\alpha\in \Aut(G)$ there exists a *-automorphism 
\begin{align*}
\Phi:W^\ast(G)\to W^\ast(G)\quad \text{such that}\quad \Phi(\pi(g))=\pi(\alpha(g)) ,\quad\text{for all}\quad g\in G.
\end{align*}
For a strongly continuous one-parameter group of automorphisms $\alpha:\R\to\Aut(G)$, the corresponding one-parameter group of *-automorphisms $\tau:\R\to \Aut(W^\ast(G))$ is $\sigma$-weakly continuous since $\spann_\C\pi(G)\subseteq W^\ast(G)$ is $\sigma$-weakly dense and on $U(\cH)$ the $\sigma$-weak and the strong operator topology coincide.
\item[(ii)] If $\varphi$ is a state of $G$ with GNS-representation $(\pi_\varphi,\cH_\varphi,\Omega_\varphi)$, then the corresponding representation $\pi_\varphi:W^\ast(G)\to \pi_\varphi(G)^{\prime\prime}$ defines the state
\begin{align*}
\omega_\varphi:W^\ast(G)\to \C,\quad A\mapsto \langle \Omega_\varphi,\pi_\varphi(A)\Omega_\varphi\rangle
\end{align*}
which satisfies $\omega_\varphi(\pi(g))=\varphi(g)$, for all $g\in G$.
\end{itemize}
}
\end{rem}
If one adapts the proof of \cite[Prop. 5.3.7]{BR96}
\footnote{In \cite[Def. 5.3.1]{BR96} KMS states of a C*-or W*-dynamical system are defined by a property for a suitably dense *-subalgebra $\mathcal{B}$ of $\tau$-analytic elements to ensure that for $A,B\in \mathcal{B}$ the maps $F_{A,B}(t):=\omega(A\tau_t(B))$ extend to $\C$ and satisfy $\omega(A\tau_{i\beta}(B))=\omega(BA)$. In the proof of \cite[Prop. 5.3.7]{BR96}, it is only used that the *-subalgebra $\cB$ consists of $\tau$-analytic elements to ensure the extension property of the functions $F_{A,B}$.}, one obtains the following lemma:
\begin{lem}\label{kmsdens}
Suppose $(\cA,\tau,\R)$ is a C*-dynamical system (resp. W*-dynamical system) and $\omega:\cA\to \C$ is a (resp. normal) state. If $\cK\subseteq \cA$ is a norm-dense (resp. $\sigma$-weakly dense) subspace (resp. *-subalgebra) such that, for $A,B\in \cK$, the functions
\begin{align*}
F_{A,B}:\R\to \C,\quad t\mapsto\omega(A\tau_t(B))
\end{align*}
extend to a continuous functions on the strip {\rm~$\ostrip:=\{z\in \C\;|\;0\leq \text{Im}(z)\leq \beta\}$} that are holomorphic on the interior and satisfy $F_{A,B}(t+i\beta)=\omega(\tau_t(A)B)$, for $t\in \R,$, then $\omega$ is a $(\tau,\beta)$-KMS state.
\end{lem}
\begin{prop}\label{kmsW*}
Let $G$ be a topological group $G$, \rm{$\alpha:\R\to\Aut(G)$} \it{be a strongly continuous one-parameter group of automorphisms and $\varphi:G\to \C$ a state. Then $\varphi$ is an $(\alpha,\beta)$-KMS state if and only if the corresponding state $\omega_\varphi:W^\ast(G)\to \C$ is a $(\tau,\beta)$-KMS state.}
\end{prop}
\begin{proof}
Observe that $\omega_\varphi(\pi(x)\tau_t(\pi(y)))=\omega_\varphi(\pi(x\alpha_t(y)))=\varphi(x\alpha_t(y))$ holds, for all $x,y\in G$. Then the assertion follows from Lemma \ref{kmsdens} since $\spann_\C\pi(G)$ is a $\sigma$-weakly dense *-subalgebra of $W^\ast(G)$ on which the KMS-condition is satisfied.
\end{proof}
\begin{cor}\label{invar}
Let $G$ be a topological group $G$, \rm{$\alpha:\R\to\Aut(G)$} \it{be a strongly continuous one-parameter group of automorphisms and $\varphi:G\to \C$ an $(\alpha,\beta)$-KMS state. Then $\varphi$ is $\alpha$-invariant.}
\end{cor}
\begin{proof}
\cite[Prop. 5.3.3]{BR96} implies that $\omega_\varphi$ is $\tau$-invariant. This proves the assertion.
\end{proof}
Let $G$ be a topological group, $\alpha:\R\to \text{Aut}(G)$ be a strongly continuous one-parameter group of automorphisms and $\varphi$ be an $\alpha$-invariant state with GNS-representation $(\pi,\cH,\Omega)$. If $(\pi_\varphi,\cH_\varphi,\Omega_\varphi)$ denotes the GNS-representation of $\varphi$, then there exists a strongly continuous unitary one-parameter group $(U_t)_{t\in \R}$ on $\cH_\varphi$ such that
\begin{align*}
U_t\pi_\varphi(g)U_{-t}=\pi_\varphi(\alpha_t(g))\quad\text{and}\quad U_t\Omega_\varphi=\Omega_\varphi,\quad\text{for}\quad t\in \R,\;g\in G.
\end{align*}
Consider the normal state $\omega_\varphi:\pi_\varphi(G)^{\prime\prime}\to \C,\; A\mapsto \langle \Omega_\varphi, A \Omega_\varphi\rangle$ and the one-parameter group of *-automorphisms $\tau$ defined on $\pi_\varphi(G)^{\prime\prime}$ by $\tau_t(A):=U_tAU_{-t}$, for $t\in \R$ and $A\in\pi_\varphi(G)^{\prime\prime}$, which is $\sigma$-weakly continuous. 
\begin{prop}
A state $\varphi$ of a group dynamical system $(G,\R,\alpha)$ is an $(\alpha,\beta)$-KMS state if and only if $\omega_\varphi$ is a $(\tau,\beta)$-KMS state of the W*-dynamical system $(\pi_\varphi(G)^{\prime\prime},\tau,\R)$.
\end{prop}
\begin{proof}
This follows immediately from the $\sigma$-weak density of $\spann_\C\pi_\varphi(G)\subseteq \pi(G)^{\prime\prime}$, the relation
\begin{align*}
\omega(\pi_\varphi(g)\tau_t(\pi_\varphi(h)))=\varphi(g\alpha_t(h)),\qfor g,h\in G,\;t\in \R
\end{align*}
and Proposition \ref{kmsdens}.
\end{proof}
\subsection{C*-methods for locally compact groups}\label{subgroupc}
Let $G$ be a locally compact group. The \textit{group $C^*$-algebra} of $G$ is defined as the universal enveloping $C^*$-algebra $C^*(G):=C^*(L^1(G))$ of the convolution algebra $L^1(G)$, i.e. the completion of $L^1(G)$ with respect to the maximal $C^*$-norm given on $L^1(G)$ by
\begin{align*}
\|f\|_{C^*}:=\sup\{\|\pi(f)\|\;|\;\pi:L^1(G)\to B( \cH)\;\text{non-degenerate *-representation}\}.
\end{align*}
Every non-degenerate *-representation $\pi:L^1(G)\to B(\cH)$ extends uniquely to a continuous *-representation of $C^\ast(G)$ and $\|\pi(f)\|\leq \|f\|_{C^\ast}$ holds trivially for all $f\in L^1(G)$. We will usually also denote this extension by $\pi$. Note that non-degenerate *-representation are in one-to-one correspondence with unitary representations of $G$ (cf. \cite[Prop. 6.2.3]{DE09}). This one-to-one correspondence yields the following one-to-one correspondence of states as they are uniquely determined by their cyclic GNS-representation and a fixed cyclic vector.
\begin{lem}\label{affin}
There is a one-to-one correspondence of states $\varphi$ of $G$ and states $\omega$ of $C^*(G)$ established by $\Phi:\cS(G)\to\cS(C^\ast(G)),\;\varphi\mapsto\omega_\varphi$, where $\omega_\varphi$ is uniquely determined by
\begin{align*}
\omega_\varphi(f)=\int_G f(x)\varphi(x)dx,\quad\text{for}\quad f\in L^1(G).
\end{align*}
If one equips $\cS(G)$ with the weak*-subspace topology from $\cS(G)\subseteq L^\infty(G)\cong L^1(G)^\prime$ and $\cS(C^\ast(G))$ with the weak*-topology of $C^\ast(G)^\prime$, then $\Phi$ is affine linear and a weak*-continuous homeomorphism.
\end{lem}
In order to get a one-to-one correspondence between KMS states of $C^*(G)$ and KMS states of $G$, one needs to define a strongly continuous one-parameter group of automorphisms 
\begin{align*}
\tau:\R\to \text{Aut}(C^*(G))\quad\text{corresponding to}\quad\alpha:\R\to\text{Aut}(G)
\end{align*} 
such that $\omega_\varphi$ is a $(\tau,\beta)$-KMS state of $C^\ast(G)$ if and only if $\varphi$ is an $(\alpha,\beta)$-KMS state of $G$. To do this, we note that the pullback $(\alpha_t)_*\mu_G$ of the Haar measure $\mu_G$ of $G$ is a left-invariant Radon measure and thus there exists a unique scalar $\lambda_{t}\in \R^+$ satisfying 
\begin{align*}
\mu_G\circ \alpha_{-t}=:(\alpha_{t})_*\mu_G=\lambda_{t}\cdot\mu_G\quad \text{for all}\quad t\in\R.
\end{align*} 
It is an easy exercise to show that $\tau:\R\times L^1(G)\to L^1(G),\; \tau_t(f):=\lambda_{t}\cdot(f\circ\alpha_{-t})$ extends to a strongly continuous one-parameter group of *-automorphisms of $C^\ast(G)$ which we also denote by~$\tau$. For $t\in \R$ and $f\in L^1(G)$, one has
\begin{align*}
\pi(\tau_t(f))=\int_G \lambda_{t}\cdot f(\alpha_{-t}(x))\pi(x)dx=\int_G f(x)\pi(\alpha_t(x))dx=(\pi\circ \alpha_t)(f).
\end{align*}
\begin{lem}\label{compstates}
Let $G$ be a locally compact topological group and \rm{$\alpha:\R\to \Aut(G)$} \it{be a strongly continuous one-parameter group of automorphisms. For a state $\varphi:G\to \C$, one has
\begin{align*}
\omega_\varphi(f\ast\tau_t(g))=\int_G\int_G f(y)\, g(x)\,\varphi(y\alpha_t(x))dxdy,\quad\text{for}\quad t\in \R,\;f,g\in L^1(G).
\end{align*}}
\end{lem}
\begin{proof}
Let $(\cH,\pi,\Omega)$ be the GNS-representation of $\varphi$ and $\omega_\varphi$. Using $\varphi(x)=\langle \Omega,\pi(x)\Omega\rangle$, we calculate $\omega_\varphi(f\ast\tau_t(g))$ which is given by
\begin{align*}
&\int_G  (f\ast \tau_t(g))(x)\langle \Omega,\pi(x)\Omega\rangle dx=\int_G\int_G f(y)\,\lambda_{t}\, g(\alpha_{-t}(y^{-1}x))\,\varphi(x)dxdy\\
=&\int_G\int_G f(y)\, \lambda_{t}\, g(\alpha_{-t}(x))\,\varphi(yx)dxdy= \int_G\int_G f(y)\, g(x)\,\varphi(y\alpha_t(x))dxdy. \qedhere
\end{align*}
\end{proof}
\begin{prop}\label{grouptoc}
A state $\varphi\in \cS(G)$ is an $(\alpha,\beta)$-KMS state if and only if $\omega_\varphi\in\cS(C^\ast(G))$ is a $(\tau,\beta)$-KMS state in the sense of $C^*$-dynamical systems. \it{In particular the map $\cS(G)\ni\varphi\to \omega_\varphi\in \cS(C^\ast(G))$ restricts to an affine linear homeomorphism of the set of respective KMS states.}
\end{prop}
\begin{proof}
Suppose $\varphi$ is an $(\alpha,\beta)$-KMS state of $G$. We first show that $\omega_\varphi$ is a $(\tau,\beta)$-KMS state. [Prop. 5.3.12, BR96] implies that it suffices to show
\begin{align*}
\int_\R \widehat{\psi}(t)\,\omega_\varphi(f\ast\tau_t(g))dt=\int_\R \widehat{\psi}(t+i\beta)\,\omega_\varphi(\tau_t(g)\ast f)dt,
\end{align*}
for all $\psi\in C^\infty_c(\R)$ and $f,g\in L^1(G)$.\\

Fix $\psi\in C_c^\infty(\R)$. The Paley-Wiener Theorem \cite[Prop. 5.3.11]{BR96} implies that, for $R>0$ and $\text{supp}(\psi)\subseteq~[-R,R]$, the function $\widehat{\psi}$ is entire analytic and for every $n\in \N_0$ there exist constants $C_n>0$ such that
\begin{align*}
|\widehat{\psi}(z)|\leq C_n(1+|z|)^{-n}\exp(R\cdot|\text{Im}(z)|),\quad\text{for}\quad z\in \C.
\end{align*} 
This proves that the map $t\mapsto \widehat{\psi}(t)\varphi(y\alpha_{t}(x))$ extends to continuous function $F$ on the strip $\overline{\cS_\beta}$ which is holomorphic on the interior. Choose, for $r>0$, a piecewise smooth curve $\gamma_r$ which parametrizes the rectangle of width depending on $r$ and fixed height $\beta$ depicted below:
\begin{align*}
\begin{tikzpicture}[scale=1.5]
\draw[->] (-1.75,0) -- (1.75,0);
\draw[->] (0,-0.5) -- (0,1.75);
\draw (-1.0,0) rectangle (1,1);
\draw[step=.25cm,gray,very thin] (-1.5,-0.5) grid (1.5,1.5);
\draw (-1 cm,-1pt) -- (-1 cm,1pt) node[anchor=north,]{$-r$};
\draw (1 cm,-1pt) -- (1 cm,1pt)node[anchor=north, ]{$r$};
\draw[thick] (-1pt, 1 cm) -- (1pt, 1 cm) node[anchor=east, above left=0.25pt]{$i\beta$};
\begin{scope}[thick, every node/.style={sloped,allow upside down}]
  \draw (-1,0)-- node {\midarrow} (1,0);
  \draw (1,0)-- node {\midarrow} (1,1);
  \draw (1,1)-- node {\midarrow} (-1,1);
  \draw (-1,1)-- node {\midarrow} (-1,0);
\end{scope}
\end{tikzpicture}
\end{align*}
Applying \cite[Prop. 5.3.5]{BR96} to the continuous function $F$ on $\overline{\cS_\beta}$ which is holomorphic on the interior yields
\begin{align*}
F(t+is\beta)\leq |\widehat{\psi}(t+is\beta)|\cdot\sup\{|\varphi(x\alpha_t(y))|,|\varphi(\alpha_t(y)x)|\}\leq C_n(1+\beta^2+t^2)^{-n}\exp (\beta R).
\end{align*}
This proves that the restriction of $F(z)$ to the right and left side of the rectangle parametrized above converges uniformly to zero for $r\to \infty$ since it is given by $|F(\pm r+is\beta)|$, for $s\in[0,1]$. The length of the left and right side of the rectangle is finite and thus the integrals over the right and left sides tend to zero for $r\to \infty$. The integrals of $F(z)$ over the contractible curves $\gamma_r$ vanish as $F$ is holomorphic which implies
\begin{align*}
0=\lim_{r\to\infty}\int_{\gamma_r}F(z)dz=\int_{-\infty}^\infty \underbrace{\widehat{\psi}(t)\varphi(y\alpha_{t}(x))}_{F(t)}dt-\int_{-\infty}^{\infty}\underbrace{\widehat{\psi}(t+i\beta)\varphi(\alpha_{t}(x)y)}_{F(t+i\beta)}dt.
\end{align*}
To see that $\omega_\varphi$ is a $(\tau,\beta)$-KMS state, we apply Lemma \ref{compstates} and obtain
\begin{align*}
&\int_\R \widehat{\psi}(t)\,\omega_\varphi(f\ast\tau_t(g))dt=\int_G\int_G f(y)\, g(x)\int_\R\widehat{\psi}(t)\,\varphi(y\alpha_{t}(x))dt dxdy\\
=&\int_G\int_Gf(y)\, g(x)\,\int_\R\widehat{\psi}(t+i\beta)\varphi(\alpha_{t}(x)y)dtdxdy
=\int_\R \widehat{\psi}(t+i\beta)\,\omega_\varphi(\tau_t(g)\ast f)dt.
\end{align*}
Suppose conversely that $\omega_\varphi$ is a $(\tau,\beta)$-KMS state. Then \cite[Prop. 5.3.4]{BR96} implies that there exists a unique $\sigma$-weakly continuous one-parameter group of automorphisms $\widehat{\tau}$ of $\pi_{\omega_\varphi}(C^\ast(G))^{\prime\prime}$ such that $\widehat{\tau}_t(\pi_{\omega_\varphi}(A))=\pi_{\omega_\varphi}(\tau_t(A))\quad\text{for all}\quad A\in C^\ast(G)$ and the state 
\begin{align*}
\widehat{\omega}_\varphi:\pi_{\omega_\varphi}(C^\ast(G))^{\prime\prime}\to\C,\;A\mapsto\langle \Omega_\varphi,A\Omega_\varphi\rangle
\end{align*} 
is a $(\widehat{\tau},\beta)$-KMS state in the sense of a W*-dynamical system. In particular
\begin{align*}
\int_G f(x) \widehat{\tau}_t(\pi_\varphi(x))dx=(\widehat{\tau}_t\circ\pi_{\omega_\varphi})(f)=\pi_{\omega_\varphi}(\tau_t(f))=\int_G f(x)\pi_\varphi(\alpha_t(x))dx,
\end{align*}
implies that $\widehat{\tau}_t(\pi_\varphi(x))=\pi_\varphi(\alpha_t(x))$, for all $x\in G$, as one can approximate with left shifts of a Dirac net. Therefore
\begin{align*}
\varphi(x\alpha_t(y))=\langle \Omega_\varphi,\pi_\varphi(x)\pi_\varphi(\alpha_t(y))\Omega_\varphi\rangle =\langle \Omega_\varphi,\pi_\varphi(x)\widehat{\tau}_t(\pi_\varphi(y))\Omega_\varphi\rangle =\widehat{\omega}_\varphi(\pi_\varphi(x)\widehat{\tau}_t(\pi_\varphi(y)))
\end{align*}
follows which proves that $\varphi$ is an $(\alpha,\beta)$-KMS state.
\end{proof}
Now we have translated all the necessary information from KMS states of groups to KMS states of a C*-algebra to prove the following theorem.
\begin{thm}\label{convex}
Let $G$ be a locally compact topological group and \rm{$\alpha:\R\to \text{Aut}(G)$} \textit{be a strongly continuous one-parameter group of automorphisms. Then the set $K_\beta\subseteq \cS(G)$ of $(\alpha,\beta)$-KMS states satisfies}
\begin{itemize}
\item[(i)] \textit{$K_\beta$ is convex and weak-$\ast$-compact with respect to the subspace topology of 
\begin{align*}
\cS(G)\subseteq L^\infty(G)\cong L^1(G)^\prime.
\end{align*}}
\item[(ii)] \textit{$K_\beta$ is a simplex} (cf. Remark \ref{simplex}).
\item[(iii)] \textit{$\varphi\in K_\beta$ is an extremal point of $K_\beta$ if and only if $\varphi$ is a factor state.}
\item[(iv)]$K_\beta=\overline{\text{Conv(Ext}(K_\beta))}$, \textit{where} $\text{Ext}(K_\beta)$ \textit{are the extremal points in $K_\beta$ and the closure is with respect to the weak-$\ast$-topology.}
\end{itemize}
\end{thm}
\begin{proof}
Lemma \ref{affin} and Proposition \ref{grouptoc} imply that we may view $K_\beta$ as the set of $(\tau,\beta)$-KMS states of $\cA:=C^\ast(G)$. Then \cite[p. 116]{BR96} shows that the $(\tau,\beta)$-KMS states are in one-to-one correspondence to the $(\tau_+,\beta)$-KMS states of the unitization $\cA_+:=\cA\oplus\C$, where $\tau_+$ acts on $\cA_+$ via 
\begin{align*}
\tau_+:\R\times \cA_+\to \cA_+,\quad (t,(A,\lambda))\mapsto(\tau_t(A),\lambda).
\end{align*} 
Then Theorem \cite[Thm. 5.3.30]{BR87} implies (i)-(iii) for the subset of $(\tau_+,\beta)$-KMS states. Since the correspondence was established via the affine linear homeomorphism with respect to the weak-$\ast$-topology given by
\begin{align*}
K_\beta\ni\omega\mapsto \omega_+:(A,\lambda)\mapsto \omega(A)+\lambda,
\end{align*} 
it follows that (i)-(iii) also hold for $K_\beta$. Then (iv) is a direct consequence of (i) and the Krein-Milman Theorem.
\end{proof}
\begin{rem}\label{simplex}
\rm{For the compact convex subset $K_\beta$ of the locally convex topological vector space $C^\ast(G)^\prime$ equipped with the weak-$\ast$-topology \cite[Subsection 4.1]{BR87} provides a decomposition theory. For a given Radon probability measure $\mu$ on $K_\beta$, \cite[Prop. 4.1.1]{BR87} proves that there exists a unique point $b(\mu)=\int_K \omega\, d\mu(\omega)\in K_\beta$ called the \textit{barycenter} of $\mu$, such that
\begin{align*}
 f(b(\mu))=\mu(f):=\int_{K}f(\omega)d\mu(\omega),
\end{align*}
for all affine linear continuous functions $f:K_\beta\to \R$. Then \cite[Thm. 4.1.15]{BR87} establishes the importance of $K_\beta$ being a simplex  since it implies that every $\omega\in K$ is the barycenter of a unique normalized Radon measure $\mu_\omega$ maximal with respect to the partial order $\succ$ defined by
\begin{align*}
\mu\succ \nu \quad :\Leftrightarrow\quad \mu(f)\geq \nu(f),\quad \text{for all}\; f:K\to \R\; \text{continuous and convex.}
\end{align*}
For separable locally compact topological groups $G$ the convolution algebra $L^1(G)$ is separable which shows that $C^\ast(G)$ is a separable C*-algebra. In particular, $K_\beta$ is metrizable and thus \cite[Thm. 4.1.11]{BR87} implies that $\mu_\omega$ is supported on the extremal points of $K_\beta$, i.e. $\mu_\omega(\text{Ext}(K_\beta))=1$. We thus obtain for each $\omega\in K_\beta$ a unique decomposition
\begin{align*}
\omega=\int_{\text{Ext}(K_\beta)}\omega^\prime d\mu_\omega(\omega^\prime).
\end{align*}
For $f\in L^1(G)$, the function $\text{Ext}(K_\beta)\times G\ni (\omega,x)\mapsto f(x)\varphi_{\omega}(x)$ is clearly integrable, the measure space $\text{Ext}(K_\beta)\times G$ is $\sigma$-finite and thus from Fubini's theorem, one obtains, for all $\omega_0\in K_\beta$, that
\begin{align*}
\omega_0(f)=\int_{\text{Ext}(K_\beta)}\int_G f(x)\varphi_\omega(x)dxd\mu_{\omega_0}(\omega)=\int_G f(x)\int_{\text{Ext}(K_\beta)}\varphi_\omega(x)d\mu_{\omega_0}(\omega)dx.
\end{align*}
This implies that the state $\varphi_{\omega_0}$ of $G$ corresponding to $\omega_0$ is given by
\begin{align*}
\varphi_{\omega_0}(x)=\int_{\text{Ext}(K_\beta)}\varphi_\omega(x)d\mu_{\omega_0}(\omega),\quad\text{for}\quad x\in G.
\end{align*}
 Since $L^1(G)\subseteq C^\ast(G)$ is dense, it follows the evaluation functionals 
\begin{align*}
\delta_f:K_\beta \to \C,\quad \omega \mapsto \omega(f),\qfor f\in L^1(G),
\end{align*} 
separate the points and thus the algebra generated by these functionals is dense in $C(K_\beta)$ with respect to the maximum norm by the Theorem of Stone-Weierstrass. This proves that $\mu_{\omega_0}$ is already uniquely determined by these functionals. As the set $K_\beta$ is convex and weak-$\ast$-closed every such measure defines an $(\alpha,\beta)$-KMS state. Hence we have proven the following proposition:
}
\end{rem}
\begin{prop}\label{bary}
Let $G$ be a separable locally compact topological group, \rm{$\alpha:\R\to \text{Aut}(G)$} \textit{be a strongly continuous one-parameter group and $\beta\neq 0$. For each $\varphi_0\in K_\beta$, there exists a Radon probability measure $\mu$ on} $\text{Ext}(K_\beta)$ \textit{uniquely determined by}
\begin{align*}
\varphi_0(x)=\int_{\text{Ext}(K_\beta)}\varphi(x)d\mu(\varphi),\quad for\; all\quad x\in G.
\end{align*}
\textit{Conversely every such measure defines an $(\alpha,\beta)$-KMS state of $G$.}
\end{prop}
\subsection{Semi-finite KMS states of topological groups}\label{subssemif}
Consider an $(\alpha,\beta)$-KMS state $\varphi$ of a topological group $G$. If $\cM_\varphi=\pi_\varphi(G)^{\prime\prime}$ is a semifinite von Neumann algebra, then there exists a unique faithful semi-finite normal trace $\tau$ on $\cM_\varphi$ (cf. \cite[Thm. 2.15]{Tak02}). Since the faithful normal state $\omega_\varphi$ is invariant under the modular automorphism group $\sigma_t^{\tau}=\text{id}_{\cM_\varphi}$ of the trace $\tau$, it follows from the Radon-Nikodym Theorem \cite[Cor. 3.6]{Tak03} that there exists a unique non-singular positive self-adjoint operator $h$ affiliated to $\cM_\varphi$ such that $\omega_\varphi(A)=\tau(hA)$, for $A\in \cM_\varphi$. Then \cite[Thm. 2.11]{Tak03} implies that the modular automorphism group of $\omega_\varphi$ is given by 
\begin{align*}
\text{Ad}(h^{-it/\beta})=\text{Ad}(e^{itH}),\qfor H=-\tfrac{1}{\beta}\log(h).
\end{align*} 
Furthermore one has $\tau(h)=\tau(e^{-\beta H})=\omega_\varphi(\mathds{1})=1$ and $\varphi$ is given by
\begin{align*}
\varphi(g)=\frac{\tau(\pi_\varphi(g)e^{-\beta H})}{\tau(e^{-\beta H})},\qfor g\in G.
\end{align*}
Suppose conversely that $\pi$ is an $\alpha$-covariant cyclic unitary representation such that $\pi(G)^{\prime\prime}$ is semi-finite and for which there exists a non-singular implementing operator $H$ of $\alpha$ affiliated to $\pi(G)^{\prime\prime}$ which satisfies $\tau(e^{-\beta H})=1$, for the unique faithful semi-finite normal trace $\tau$ on $\cM_\varphi=\pi(G)^{\prime\prime}$. We want to show that this defines an $(\alpha,\beta)$-KMS state of $G$ by constructing a $(\sigma,\beta)$-KMS state, where $\sigma_t(A):=e^{itH}Ae^{-itH}$ for $A\in \cM_\varphi$. To do this, observe that \cite[Thm. 2.22]{Tak02} implies that the semi-finite von Neumann algebra $\cM_\varphi\otimes \cM_\varphi^\text{op}$ acts on $L^2(\cM_\varphi,\tau)$, which is defined as the completion of the complex pre-Hilbert space
\begin{align*}
\{A\in \cM_\varphi\;|\;\tau(|A|^2)<\infty\}
\end{align*}
with respect to the scalar product $\langle A,B\rangle:=\tau(A^\ast B)$, by left and right multiplication. If one denotes 
\begin{align*}
\lambda:\cM_\varphi\to B(L^2(\cM_\varphi,\tau)),\quad &\lambda(a)x=ax,\qfor \tau(|x|^2)<\infty\\
\text{and}\quad \rho:\cM_\varphi^\text{op}\to B(L^2(\cM_\varphi,\tau)),\quad &\rho(a)x=xa,\qfor \tau(|x|^2)<\infty,
\end{align*} 
one has that $\lambda(\cM_\varphi)^\prime =\rho(\cM_\varphi)^{\prime\prime}$ and $\rho(\cM_\varphi)^\prime =\lambda(\cM_\varphi)^{\prime\prime}$ (cf.  \cite[Thm. 2.22]{Tak02}).\\

In particular if $\Omega=e^{-\tfrac{1}{2}\beta H}$, then $\Omega\in L^2(\cM_\varphi,\tau)$ is non-singular and self-adjoint. This implies that the vector $\Omega$ is cyclic and separating for $\cM_\varphi$ because by the Range-Kernel Lemma one has $\overline{\cR(\Omega)}=\text{ker}(\Omega^\ast)=\text{ker}(\Omega)=~\{0\}$ and therefore $A\Omega=0$ implies $A=0$ and $\Omega B=0$ implies $(\Omega B)^\ast=B^\ast\Omega=0$, i.e. $B^\ast=0$ and thus $B=0$, for $A, B\in \cM_\varphi$.\\

The corresponding modular objects $(\Delta,J)$ are given on $\lambda(\cM)\Omega$ by $J(A\Omega)=\Omega A^\ast$ and $\Delta(A\Omega)=\Omega A$, for $A\in \Omega$ and thus the corresponding modular automorphism group is given by
\begin{align*}
\sigma_t:\cM\to \cM,\quad A\mapsto e^{itH}Ae^{-itH},\qfor t\in \R.
\end{align*}
This implies that the state
\begin{align*}
\omega:\cM\to \C,\quad A\mapsto \frac{\tau(Ae^{-\beta H})}{\tau(e^{-\beta H})}=\frac{\langle \Omega,A\Omega\rangle}{\|\Omega\|^2}
\end{align*} 
is a $(\sigma,\beta)$-KMS state. This proves that $\varphi:G\to \C,\;g\mapsto \omega(\pi(g))$ is an $(\alpha,\beta)$-KMS state. Hence one has the following proposition:
\begin{prop}
Let $G$ be a topological group, {\rm~$\alpha:\R\to \text{Aut}(G)$} be a strongly continuous one-parameter group of automorphisms. There is a one-to-one correspondence between $(\alpha,\beta)$-KMS states $\varphi$ such that $\pi_\varphi(G)^{\prime\prime}$ is a semi-finite von Neumann algebras and $\alpha$-covariant cyclic unitary representation $\pi$ such that $\pi_\varphi(G)^{\prime\prime}$ is a semi-finite for which there exists an implementing operator $H$ of $\alpha$ affiliated to $\pi(G)^{\prime\prime}$ which satisfies $\tau(e^{-\beta H})=1$, for the unique faithful semi-finite normal trace $\tau$ on $\pi(G)^{\prime\prime}$.
\end{prop}
\begin{rem}\label{typeI}
\rm{If $\varphi$ is a factorial type I $(\alpha,\beta)$-KMS state, then $\cM_\varphi=\pi_\varphi(G)^{\prime\prime}$ is a type I von Neumann algebra and therefore there exists a Hilbert space $\cH$ and an isomorphism of von Neumann algebras $\Phi:\cM_\varphi\to B(\cH)$. The unitary representation $\rho:=\Phi\circ \pi_\varphi$ is irreducible since the $\sigma$-weak closure of $\spann_\C(\rho(G))$ is given by $\rho(G)^{\prime\prime}=\Phi(\cM_\varphi)=B(\cH)$. This proves that there exists an operator $H$ on $\cH$ such that 
\begin{align*}
e^{itH}\rho(g)e^{-itH}=\rho(\alpha_t(g))\qand\tr(e^{-\beta H})<\infty
\end{align*} 
since the GNS-representation of $\varphi$ is $\alpha$-covariant for the modular automorphism group, i.e. $\Delta^{it}\pi_\varphi(g)\Delta^{-it}=\pi_\varphi(\alpha_t(g))$ and $\Phi\circ\pi_\varphi=\rho$. In particular
\begin{align*}
\varphi:G\to \C,\quad \frac{\tr(\rho(g)e^{-\beta H})}{\tr(e^{-\beta H})},\qfor g\in G,
\end{align*} 
is a Gibbs state corresponding to the irreducible unitary representation $\rho$ of $G$. Since the representation $\rho$ is unique, the implementing operator $H$ is unique up to the addition of a real scalar.
}
\end{rem}
\begin{rem}\label{extend}
\rm{By \cite[Thm. 3.14]{Tak03} a von Neumann algebra $\cM$ is semi-finite if and only if for every faithful semi-finite normal weight $\psi$, the corresponding modular automorphism group $\sigma_t^\psi$ is inner in the sense that there exists a self-adjoint operator $H$ affiliated to $\cM$ such that $\sigma_t^\psi=\text{Ad}(e^{itH})$ as automorphisms of $\cM$. For KMS states of topological groups, this innerness of the modular automorphism group translates to the extension of the factorial KMS state to the larger group $G^\flat:=G\rtimes_\alpha\R$. Let $\varphi$ be an $(\alpha,\beta)$-KMS state such that $\pi_\varphi(G)^{\prime\prime}$ is a semi-finite von Neumann algebra and $H$ the corresponding implementing operator affiliated to $\cM_\varphi:=\pi_\varphi(G)^{\prime\prime}$ satisfying $\tau(e^{-\beta H})<\infty$. It follows that the unitary representation 
\begin{align*}
\pi_\varphi^{\flat}:G^\flat=G\rtimes_\alpha \R\to U(\cH),\quad (g,t)\mapsto \pi_\varphi(g)e^{itH}
\end{align*}
is a semi-finite representation since $e^{itH}\in \cM_\varphi$ implies $\pi_\varphi^{\flat}(G\rtimes_\alpha \R)^{\prime\prime}=\pi_\varphi(G)^{\prime\prime}=\cM_\varphi$. In particular 
\begin{align*}
\varphi^\flat:G^\flat=G\rtimes_\alpha\R\to \C,\quad (g,t)\mapsto\frac{\tau(\pi_\varphi^\flat(g,t)e^{-\beta H})}{\tau(e^{-\beta H})}
\end{align*}
is an $(\alpha^\flat,\beta)$-KMS state of $G^\flat=G\rtimes_\alpha\R$, where $\alpha^\flat$ is the one-parameter group of inner automorphisms of $G^\flat$ defined by conjugation with $(e,t)$, i.e.
\begin{align*}
\alpha^\flat(g,s)=(e,t)(g,s)(e,t)^{-1}=(\alpha_t(g),s),\qfor g\in G,\;s,t\in \R.
\end{align*}
Note that the modular automorphism group which acts on $\cH_\varphi \cong L^2(\cM_\varphi,\tau)$ by conjugation with $e^{itH}$ and thus leaves the cyclic vector $\Omega =\frac{ e^{-\tfrac{1}{2}\beta H}}{\sqrt{\tr(e^{-\beta H})}}$ invariant, but the left multiplication by $e^{itH}$ does not. 
}
\end{rem}
\begin{thm}\label{classconj}
Let $\beta>0$ and suppose $G$ is a finite dimensional Lie group and \rm{$\alpha:\R\to\text{Aut}(G)$} \it{is given by $\alpha_t(g)=\exp(tX)g\exp(-tX)$, for some fixed {\rm~$X\in \g:=\textbf{L}(G)$} and $g\in G$. Then the extremal type I $(\alpha,\beta)$-KMS states are the extremal Gibbs states}\rm{
\begin{align*}
\varphi(g)=\frac{\text{tr}(\rho(g)e^{i\beta\del\rho(X)})}{\text{tr}(e^{i\beta\del\rho(X)})},
\end{align*}}\it{where $\rho$ is an irreducible representation of $G$ such that $e^{i\beta\del\rho(X)}$ is of trace class. In particular, if $G$ is simply connected, then there is a one-to-one correspondence between} 
\begin{itemize}
\item[{\rm~(i)}] extremal type I $(\alpha,\beta)$-KMS states
\item[{\rm~(ii)}] equivalence classes of irreducible unitary highest weight representations of admissible quotients $q:\g \twoheadrightarrow\g_1$ with respect to an admissible positive system $\Delta_1^+$ of $\g_1$ such that $q(X)$ is conjugate to \rm{$C_\text{max}(\Delta_{1,p}^+)^\circ$.}
\end{itemize}
\end{thm}
\begin{proof}
This follows immediately from Remark \ref{typeI} since $\rho(\exp(tX))=e^{t\del\rho(X)}=e^{it(-i\del\rho(X))}$ implies that the corresponding implementing operator $H$, which is unique up to addition by a real scalar, is w.l.o.g. given by $-i\del\rho(X)$. Therefore $e^{i\beta\del\rho(X)}$ is of trace class and Corollary \ref{class} applies if $G$ is simply connected.
\end{proof}
\begin{rem}
\rm{If $G$ is a finite dimensional Lie group and the action $\alpha:\R\to \text{Aut}(G)$ is not inner, then one considers $G^\flat=G\rtimes_\alpha\R$ with Lie algebra $\g^\flat:=\textbf{L}(G^\flat)$. In order to classify all representations of Gibbs type, it suffices to classify all admissible quotients $\g^\prime$ of $\g^\flat$ and check if the image of $(0,1)\in \g^\flat=\g\rtimes \R$ under the quotient map $q:\g^\flat\to\g^\prime$ lies in the interior of $C_\text{max}(\Delta_p^+)^\circ$, for an admissible positive system $\Delta^+$ of $\g^\prime$. If one then classifies the unitary highest weight representations of $\g^\prime$ with respect to $\Delta^+$, one obtains in the same manner as above a representation of Gibbs type of $G$, and all representations of Gibbs type are obtained in this manner.

}
\end{rem}
\section{Perspectives and open problems}
\subsection{Extensions of KMS states of $G$ to KMS states of $G^\flat$}
To completely understand how the extension of a topological group $G$ to $G^\flat$ affects the study of KMS states, it would be interesting to know if every $(\alpha,\beta)$-KMS state of $G$ extends to an $(\alpha^\flat,\beta)$-KMS state of $G^\flat=G\rtimes_\alpha\R$ and if every KMS state of $G^\flat$ arises in this manner. We have seen in Remark \ref{extend} that semi-finite $(\alpha,\beta)$-KMS states of  $G$ extend to $(\alpha^\flat,\beta)$-KMS states of $G^\flat$. For a not necessarily semi-finite $(\alpha,\beta)$-KMS state $\varphi$ of $G$ with GNS-representation $(\pi_\varphi,\cH_\varphi,\Omega_\varphi)$ and $\cM:=\pi_\varphi(G)^{\prime\prime}$, it is not clear if such KMS extensions always exist. It might even be possible that this extension property already implies that the state is semi-finite.  \\ 

We considered the crossed product $\cM\rtimes_\sigma \R$, where $\sigma:\R\to \Aut(\cM)$ is the modular automorphism group of the associated state $\omega_\varphi$ of $\cM\subseteq B(\cH_\varphi)$ (cf. \cite[Section X.1]{Tak03}). This crossed product is defined as the von Neumann algebra generated by the sets of bounded linear operators on $L^2(\R,\cH_\varphi)$ given by
\begin{align*}
(\pi_\sigma(A)f)(t):=\sigma_{-t}(A)f(t) \qand (\lambda(s)f)(t):=f(t-s),
\end{align*}
for $A\in \cM, \;s,t\in \R,\;f\in L^2(\R,\cH_\varphi)$. Then one can consider the associated dual weight $\widetilde{\omega}_\varphi$ of $\cM\rtimes_\sigma \R$ and \cite[Thm. X.1.17(ii)]{Tak03} implies that the modular automorphism group $\widetilde{\sigma}:\R\to \Aut(\cM\rtimes_\sigma\R)$ of the dual weight $\widetilde{\omega}_\varphi$ satisfies
\begin{equation}\label{eqcross}
\widetilde{\sigma}_t(\pi_\sigma(A))=\pi_\sigma(\sigma_t(A)),\qand \widetilde{\sigma}_t(\lambda(s))=\lambda(s),\qfor A\in  \cM,\;s,t\in \R.
\end{equation}
For the unitary representation $\rho$ of $G^\flat $ on $L^2(\R,\cH_\varphi)$ defined as $\rho(g,t):=\pi_\sigma(\pi_\varphi(g))\lambda(t)$, the relation $\pi_\varphi(\alpha_t(g)))=\sigma_t(\pi_\varphi(g))$ and (\ref{eqcross}) imply
\begin{align*}
\rho(\alpha_t^\flat(g,s))=\rho(\alpha_t(g),s)=\widetilde{\sigma}_t\left(\rho(g,s)\right),\qfor g\in G,\;s,t\in \R.
\end{align*}
But dual weight $\widetilde{\omega}_\varphi$ of a state does not define a state of the crossed product and thus one cannot naively define the extension of $\varphi$ as $
\varphi^\flat(g,t):=\widetilde{\omega}_\varphi(\rho(g,t))$ for $g\in G,\;t\in \R$. This does not work as the dual weight does not define a state. To see that, observe that \cite[Thm. X.1.17(i)]{Tak03} implies that, for a compactly supported $\cM$-valued continuous function $f\in C_c(\R,\cM)$ and for the representation $\widetilde{\pi}_\sigma$ of $C_c(\R,\cM)$ on $L^2(\R,\cH_\varphi)$ defined as
\begin{align*}
\widetilde{\pi}_\sigma(f):=\int_\R \lambda(s)\pi_\sigma(f(s))ds,\quad \text{one has}\quad \widetilde{\omega}_\varphi(\widetilde{\pi}_\sigma(f))=\omega_\varphi(f(e))=\langle \Omega_\varphi,f(e)\Omega_\varphi\rangle.
\end{align*}
As $\widetilde{\pi}_\sigma(C_c(\R,\cM))$ generates $\cM\rtimes_\sigma\R$ (cf. \cite[Lemma X.1.10(iii)]{Tak03}), this implies that the dual weight does not define a state of the cross product $\cM\rtimes_\sigma\R$ since delta distributions do not define states as the passage from $C_c(\R,\cM)$ to $L^\infty(\R,\cM)$ suggests. Therefore this attempt does not seem to work but it illustrates a few key structures.
\subsection{Factorial Type II KMS states of Lie groups}
We are not aware of how type II factor representations $\pi$ of finite dimensional Lie groups $G$ look like. For us, the more natural context of type II factor representation is given by inductive limits of Lie groups which then act for example an the hyperfinite $\text{II}_1$ factor defined as a infinite product of two by two matrix algebras (cf. \cite{Po67}). For KMS states of the inductive limit group $U(\infty)=\varinjlim_n U(n)$ this has been done in \cite{St78} whereas general factor representations of $U(\infty)$ are studied in \cite{SV75}. A survey of these results is given in \cite{SV82}. Nonetheless it would still be interesting to know what conditions does $\tau(e^{i\beta \del\pi(X)})<\infty$ implies for finite dimensional Lie groups $G$, representation $\pi$ and generator $X\in \g$? Clearly, the spectrum of $H=-i\del\pi(X)$ does not need to be semi-bounded anymore as in type II factors there exists a sequence of non-zero mutually orthogonal projections $\{p_n\}_{n\in \N}$ such that $0<\tau(p_n)\to 0$. Instead of the semi-boundedness of $\text{spec}(H)$ on has that the spectrum decreases exponentially fast in the sense that $\int_\R e^{-\beta x}d\tau(P_H(x))<\infty$ holds, i.e. $\tau(P_H((-\infty,t_0]))\leq e^{\beta t_0}\tau(e^{-\beta H})$ goes to $0$ exponentially as $t\to -\infty$.

\appendix
\section{Appendix}

\subsection{Compactly embedded Cartan subalgebras and admissible Lie algebras}\label{subsadm}
\begin{defin}\label{compemb}
\rm{We call a subalgebra $\a\subseteq\g$ \it{compactly embedded}} \rm{if $\overline{\langle e^{\ad_\g(\a)}\rangle}\subseteq \text{GL}(\g)$ is compact. We call an element $x\in \g$ elliptic if the subalgebra $\R\cdot x\subseteq \g$ is compactly embedded and denote with $\text{comp}(\g)$ the subset of elliptic elements in $\g$.}
\end{defin}
\begin{rem}\label{type}\rm{(cf. \cite[Thm. VII.2.2]{NE00}):\\
If $\t$ is a compactly embedded Cartan subalgebra, then it is nilpotent and compact, i.e. abelian. Therefore the subset $\ad_{\g_\C}(\t_\C)\subseteq\mathfrak{gl}(\g_\C)$ is an abelian subalgebra consisting of diagonalizable elements since the compact group $\text{INN}_\g(\t)$ is a compact subgroup of the compact subgroup O$(\g,\langle\cdot,\cdot\rangle)$ of $\mathfrak{gl}(\g)$ for some scalar product on $\g$. Hence one obtains a root decomposition
\begin{align*}
\g_\C=\t_\C\oplus\bigoplus_{\alpha\in\Delta}\g_\C^\alpha,
\end{align*}
where $\g_\C^\alpha:=\{X\in\g_\C\;|\;[t,X]=\alpha(t)X,\;\text{for all}\; t\in\t_\C\}$ and $\Delta:=\{\alpha\in\t_\C^\ast\;|\;\g_\C^\alpha\neq \{0\}\}$. On the complex Lie algebra $\g_\C$ we define the involutive antihomomorphism
\begin{align*}
(X+iY)^\ast:=-X+iY\quad\text{for}\quad X,Y\in\g
\end{align*}
and observe that $\g=\{X\in \g_\C\;|\;X^\ast=-X\}$. It follows from \cite[Thm. VII.2.2]{NE00} that for $Z\in\g_\C^\alpha$, one has $Z^\ast\in \g_\C^{-\alpha}$ and $[Z,Z^\ast]\in i\t$. Define 
\begin{align*}
\g_\C(Z):=\text{span}_\C\{Z,Z^\ast,[Z,Z^\ast]\}\quad\text{and}\quad\g(Z):=\g_\C(Z)\cap\g
\end{align*} 
For roots $\alpha\in\Delta$ , one now distinguishes between the following cases:
\begin{itemize}
\item[(NS)] There exists an element $Z\in\g_\C^\alpha$ such that $\alpha([Z,Z^\ast])<0$ (the non-compact simple type). Then $\g(Z)\cong \mathfrak{sl}(2,\R)$.
\item[(CS)] There exists an element $Z\in\g_\C^\alpha$ such that $\alpha([Z,Z^\ast])>0$ (the compact simple type). Then $\g(Z)\cong \mathfrak{su}(2)$.
\item[(N)]  $\alpha([Z,Z^\ast])=0$ for all $Z\in \g_\C^\alpha$ and there exists $Z\in \g_\C^\alpha$ such that $[Z,Z^\ast]\neq 0$ (the nilpotent type). Then $\g(Z)$ is isomorphic to the three-dimensional Heisenberg algebra $\h_3(\R)$.
\item[(A)] $[Z,Z^\ast]=0$ for all $Z\in \g_\C^\alpha$ (the abelian type). Then $\g(Z)\cong \R^2$.
\end{itemize}
That only these four cases occur follows from \cite[Lemma VII.2.3]{NE00} which implies $\text{dim}_\C\g_\C^\alpha=1$, for all $\alpha\in\Delta$ of simple type, and thus conditions (NS) and (CS) hold for all $Z\in\g_\C^\alpha$ if they hold for one.
}
\end{rem}
\begin{defin}\label{defroots}
\rm{(cf. \cite[Def. VII.2.4]{NE00} and \cite[Def. VII.2.6]{NE00}):\\
Let $\t\subseteq\g$ be a compactly embedded Cartan subalgebra. }
\begin{itemize}
\item[(i)] \rm{We call a root $\alpha\in \Delta$ \textit{semisimple} if it is of simple type and denote the set of semisimple roots by $\Delta_s$. The set of \textit{solvable roots} is denoted by $\Delta_r:=\Delta\setminus\Delta_k$.}
\item[(ii)]\rm{ We call a root $\alpha\in \Delta$ \it{compact}}\rm{ if there exists $Z\in\g_\C^\alpha$ such that $\alpha([Z,Z^\ast])<0$. All other roots are called \textit{non-compact} and the set of all non-compact roots is denoted with $\Delta_p$ and the set of compact roots as $\Delta_k$. We also define the non-compact semisimple roots $\Delta_{p,s}:=\Delta_p\cap \Delta_s$.}
\item[(iii)]\rm{ A subset $\Delta^+\subseteq \Delta$ is called a \it{positive system}}\rm{ if there exists a regular element $X_0\in i\t$ such that 
\begin{align*}
\Delta^+=\{\alpha\in \Delta\;|\;\alpha(X_0)>0\}.
\end{align*}
If $\Delta^+$ is such a system of positive roots we define $\Delta_k^+,\Delta_p^+,\Delta_{p,s}^+$ and $\Delta_r^+$ as the intersection of the respective subset with $\Delta^+$.}
\item[(iv)]\rm{We call a positive system $\Delta^+$} \it{adapted}\rm{ if for all $\alpha\in \Delta_k$ and $\beta\in \Delta_p^+$ one has 
\begin{align*}
\beta(X_0)>\alpha(X_0),
\end{align*} 
for some $X_0$ defining $\Delta^+$.

}
\end{itemize}

\end{defin}
\begin{defin}\label{cone}
\rm{(cf. \cite[Def. VIII.2.6]{NE00}):\\
Let $\t\subseteq\g$ be a compactly embedded Cartan subalgebra and $\Delta^+$ be a corresponding system of positive roots. We associate to the positive system of non-compact roots $\Delta_p^+$ the convex cones
\begin{align*}
C_\text{min}(\Delta_p^+)&:=\text{cone}(\{i[Z_\alpha,Z_\alpha^\ast]\;|\;Z_\alpha\in\g_\C^\alpha,\;\alpha\in\Delta_p^+\}\subseteq \t\\
\text{and}\quad C_\text{max}(\Delta_p^+)&:=(i\Delta_p^+)^\star:=\{X\in\t\;|\;i\alpha(X)\geq 0\;\text{for all}\;\alpha\in\Delta_p^+\}.
\end{align*}
If it is clear which system of positive non-compact roots $\Delta_p^+$ is considered, then we simply write $C_\text{min}$, resp. $C_\text{max}$, for $C_\text{min}(\Delta_p^+)$, resp. $C_\text{max}(\Delta_p^+)$.
}
\end{defin}
\begin{defin}\label{defadm}
\rm{We call a Lie algebra $\g$ \textit{admissible} if $\g$ contains an invariant closed generating convex subset $C$ which does not contain any affine line.}
\end{defin}
\begin{thm}\label{thmsandwich}
\rm{(\textbf{Sandwich Theorem} \cite[Thm. VII.3.8]{NE00})}\\
\textit{Let $\t\subseteq \g$ be a compactly embedded Cartan subalgebra and $\emptyset\neq C\subseteq \g$ a closed generating invariant elliptic convex subset. Then there exists a uniquely determined adapted positive system $\Delta_p^+$ of non-compact roots such that}\rm{$C\cap \t\subseteq C_\text{max}(\Delta_p^+).$} \it{Moreover, we also have} \rm{$C_\text{min}(\Delta_p^+)\subseteq \lim (C\cap \t)$.}\it{ If, in addition, $C$ is a cone, then this means that}\rm{
\begin{align*}
C_\text{min}(\Delta_p^+)\subseteq C\cap \t \subseteq C_\text{max}(\Delta_p^+).
\end{align*}
}
\end{thm}

\subsection{Convex geometric and representation-theoretic concepts}\label{subsrepr}
\begin{defin}\label{defconv}
\rm{(cf. \cite[Def. V.1.1]{NE00} and \cite[Prop. V.1.6]{NE00}):\\
Let $V$ be a finite-dimensional real vector space and $W\subseteq V$ be a closed convex cone.
\begin{itemize}
\item[(i)] For a subset $E\subseteq V$, we put
\begin{align*}
B(E):=\{\omega\in V^\ast\;|\;\inf\omega(E)>-\infty\}\quad\text{and}\quad E^\star:=\{\omega\in V^\ast\;|\;\omega(E)\subseteq[0,\infty)\}.
\end{align*}
\item[(ii)] We define, for a convex subset $C\subseteq V$, the \textit{recession cone} $\lim(C):=\{v\in V\;|\;v+C\subseteq C\}$.
\item[(iii)] For $I\subseteq V^\ast$, we define $I^\bot:=\{v\in V\;|\;\omega(v)=0\;\text{for all}\;\omega\in I\}$.
\item[(iv)] For a subset $E\subseteq V$, we define $\text{algint}(E)$ as the interior of $E$ with respect to the affine subspace spanned by $E$.
\end{itemize}

}
\end{defin}
\begin{rem}\label{rem}
\rm{
We remark the following:
\begin{itemize}
\item[(i)] $E^\star$ is always closed but $B(E)$ is not closed in general.
\item[(ii)] One has $B\left(\overline{\text{conv}(E)}\right)=B(E)$.
\item[(iii)] $B(E)^\star=\lim(E)$ for $E$ convex (cf. \cite[Prop. V.1.14]{NE00})
\end{itemize}
}
\end{rem}
\begin{defin}\label{convmom}
\rm{(cf. \cite[Def. X.1.2]{NE00}):\\
Let $(\pi,\cH)$ be a unitary representation of a Lie group $G$ with Lie algebra $\textbf{L}(G)=\g$ on the Hilbert space $\cH$ and let $\cH^\infty$ be the dense subspace of $\cH$ of smooth vectors on which $\g$ acts as $d\pi(X)(v):=\difftev \pi(\exp_G(tX))v$ for $X\in \g,\quad v\in \cH^\infty$. Then we define the \textit{momentum map} of the unitary representation $\pi$ as
\begin{align*}
\Phi_\pi: \mathds{P}(\cH^\infty):=(\cH^\infty\setminus\{0\})/{\C^\times}\to \g^\ast,\quad \Phi_\pi([v])(X):=-i\frac{\langle v,d\pi(X)v\rangle}{\langle v,v\rangle}
\end{align*}
and the \textit{momentum set} $I_\pi$ as the closed convex hull of the image of $\Phi_\pi$.
}
\end{defin}
\begin{rem}
\rm{If $(\pi,\cH)$ is a unitary representation of a Lie group $G$ we consider the cone 
\begin{align*}
B(I_\pi)=\{X\in \g\;|\;\inf\langle I_\pi,X\rangle >-\infty\}.
\end{align*}
This cone is invariant since $I_\pi$ is invariant under the coadjoint action (cf. \cite[Lemma X.1.3]{NE00}) and \cite[Lemma X.1.6]{NE00} implies that
\begin{align*}
B(I_\pi)=\{X\in \g\;|\;\sup(\text{spec}(i\cdot d\pi(X)))<\infty\}.
\end{align*}
}
\end{rem}

\begin{defin}
\rm{ We call a pair $(\g,\ast)$ an \textit{involutive Lie algebra} if $\g$ is a complex Lie algebra and $\ast:\g\to\g$ is an antilinear antiautomorphism, i.e. $[x^\ast,y^\ast]=-[x,y]^\ast$, holds for all $x,y\in\g$.
}
\end{defin}

\begin{defin}
\rm{We say that an involutive Lie algebra $(\g,\ast)$ has a \textit{root decomposition} if there exists a $\ast$-invariant subalgebra $\h\subseteq\g$ and a subset $\Delta\subseteq\g^\ast$ such that $\alpha(h^\ast)=\overline{\alpha(h)}$ and
\begin{align*}
\g=\h\oplus\bigoplus_{\alpha\in \Delta}\g^\alpha,
\end{align*}
for $\g^\alpha:=\{X\in\g\;|\;\ad(h)(X)=\alpha(h)X\;\text{for all}\; h\in\h\}\neq \{0\}$.
}
\end{defin}

\begin{defin}
\rm{Let $(\g,\ast)$ be an involutive Lie algebra. We call a $\g$-module $V$ \textit{unitary} if there exists a positive definite hermitian form $h:V\times V\to \C$ that is also contravariant, i.e. it satisfies $h(X.v,w)=h(v,X^\ast.w)$ for all $X\in\g$ and $v,w\in V$.
}
\end{defin}
\begin{defin}
\rm{Let $(\g,\ast)$ be an involutive Lie algebra with root decomposition and $\Delta^+\subseteq \Delta$ be a positive system. We call a $\g$-module $V$ a \textit{highest weight module} with respect to $\Delta^+$ if there exists a linear functional 
\begin{align*}
\lambda:\b:=\h\oplus\bigoplus_{\alpha\in \Delta^+}\g^\alpha\to \C
\end{align*}
and a cyclic vector $v\in V$ such that $X.v=\lambda(X)v$ holds for all $X\in\b$. We call $\lambda$ the \textit{highest weight} of $V$ and $v$ a \textit{highest weight vector.}

}
\end{defin}

\begin{defin}\label{unithigh}
\rm{An irreducible unitary representation $(\pi,\cH)$ of a connected Lie group $G$ is called a \textit{unitary highest weight representation} if the Lie algebra $\g$ of $G$ contains a compactly embedded Cartan subalgebra $\t\subseteq\g$ and if $\k\subseteq \g$ is the unique maximal compactly embedded subalgebra containing $\t$ (cf. \cite[Prop. VII.25]{NE00}) and $K:=\exp(\k)$, then the space $\cH^{K,\infty}$ of smooth $K$-finite vectors is a highest weight module for the involutive Lie algebra $\g_\C$ with involution $(X+iY)^\ast=-X+iY$, for $X,Y\in \g$.

}
\end{defin}

\end{document}